\documentclass{amsart}
\usepackage{amssymb}
\usepackage{mathabx}
\usepackage{mathtools}
\usepackage{hyperref}

\usepackage[T1]{fontenc}
\usepackage[utf8]{inputenc}
\usepackage{lmodern}

\usepackage{bbm}
\usepackage[capitalise]{cleveref}

\usepackage{tikz}
\usetikzlibrary{decorations.pathreplacing,intersections,calc}

\usepackage{bbold}
\DeclareSymbolFont{bbold}{U}{bbold}{m}{n}
\DeclareSymbolFontAlphabet{\mathbbold}{bbold}

\usepackage{color,soul}
\definecolor{ItalianApricot}{rgb}{1,0.7,0.5}
\sethlcolor{ItalianApricot}

\newcommand{\newauthor}[2]{
\definecolor{#1}{rgb}{#2}
\expandafter\newcommand\csname #1\endcsname[1]%
{{\sethlcolor{#1}\hl{#1 says ``##1''}}}}

\newauthor{Joe}{0.8,0.8,1} % Periwinkle
\newauthor{Noam}{0.8,0.9,0.7}

\theoremstyle{plain}
\newtheorem{thm}{Theorem}[section]
\newtheorem{theorem}[thm]{Theorem}
\newtheorem{prop}[thm]{Proposition}
\newtheorem{proposition}[thm]{Proposition}
\newtheorem{lem}[thm]{Lemma}
\newtheorem{lemma}[thm]{Lemma}
\newtheorem{cor}[thm]{Corollary}

\newtheorem{fact}[thm]{Fact}

\theoremstyle{definition}
\newtheorem{defn}[thm]{Definition}
\newtheorem{definition}[thm]{Definition}
\newtheorem{example}[thm]{Example}

\theoremstyle{remark}
\newtheorem{remark}[thm]{Remark}
\newtheorem{claim}{Claim}[thm]

\numberwithin{equation}{section}

\renewcommand{\epsilon}{\varepsilon}
\renewcommand{\phi}{\varphi}
\newcommand{\seq}[1]{{\left\langle{#1}\right\rangle}}
\newcommand\+[1]{\mathcal{#1}}

\DeclareMathOperator{\dom}{dom}

\newcommand{\R}{\mathbb{R}}

\newcommand{\1}{\mathbbm{1}}

\newcommand{\tth}{{}^{\textup{th}}}
\newcommand{\conc}{\hat{\,\,}}
\newcommand{\andd}{\,\,\,\&\,\,\,}

\newcommand{\converge}{\!\!\downarrow}
\newcommand{\diverge}{\!\!\uparrow}

\renewcommand{\setminus}{\smallsetminus}
\newcommand{\w}{\omega}
\newcommand{\s}{\sigma}
\newcommand{\vphi}{\varphi}

\renewcommand{\le}{\leqslant}
\renewcommand{\ge}{\geqslant}
\renewcommand{\leq}{\leqslant}
\renewcommand{\geq}{\geqslant}
\renewcommand{\preceq}{\preccurlyeq}
\renewcommand{\succeq}{\succcurlyeq}
\renewcommand{\npreceq}{\npreccurlyeq}

\newcommand{\nle}{\nleqslant}

\newcommand{\Tur}{\textup{\scriptsize T}}
\newcommand{\leT}{\le_{\Tur}}

\newcommand{\abs}[1]{\|#1\|}
\newcommand{\symdif}{\!\vartriangle\!}

\newcommand{\cost}{\mathbf{c}}
\newcommand{\cc}{\mathbf{c}}

\newcommand{\version}[2]{#1\langle{#2}\rangle}
\newcommand{\limcost}{\underline{\cost}}

\newcommand{\dcost}{\mathbf{d}}
\newcommand{\dlimcost}{\underline{\dcost}}

\newcommand{\ecost}{\mathbf{e}}
\newcommand{\elimcost}{\underline{\ecost}}

\newcommand{\then}{\rightarrow}

\newcommand{\density}[2]{\varrho (#1 \!\mid\! #2 )}

\newcommand{\norm}[1]{\| \+{#1} \|}

\newcommand{\zeroVector}{\mathbbold{0}}
\newcommand{\oneVector}{\mathbbold{1}}

\newcommand{\emptystring}{{\seq{}}} % Joe likes \lambda; Noam is happy with \seq{}, \emptyset, or anything really
\newcommand{\leb}{\mu}
\DeclareMathOperator{\upto}{\upharpoonright}

\newcommand{\rest}[1]{\upto{#1}}

\newcommand{\DII}{\Delta^0_2}
\newcommand{\NN}{{\mathbb{N}}}
\newcommand{\RR}{{\mathbb{R}}}

\newcommand{\ZZ}{{\mathbb{Z}}}
\newcommand{\sub}{\subseteq}

\newcommand{\ML}{Martin-L{\"o}f}

\newcommand{\Halt}{{\ES'}}
\newcommand{\ES}{\emptyset}

\newcommand{\ria}{\rightarrow}
\newcommand{\tp}[1]{2^{#1}}

\newcommand{\fa}{\forall}
\newcommand{\lep}{\le^+}

\newcommand{\Opcl}[1]{[#1]^\prec}
\newcommand{\Om}{\Omega}

\begin{document}

% \title[A dense hierarchy of subideals of the $K$-trivial degrees]{Computing from projections of random points:\\A dense hierarchy of subideals\\of the \boldmath $K$-trivial degrees}

\title{Computing from projections of random points}

\date{\today}

\author{Noam Greenberg}
\address{School of Mathematics, Statistics and Operations Research\\
Victoria University of Wellington\\
Wellington, New Zealand}
\email{greenberg@msor.vuw.az.nz}

\author{Joseph S.~Miller}
\address{Department of Mathematics\\
University of Wisconsin\\
Madison, WI 53706, USA}
\email{jmiller@math.wisc.edu}

\author{Andr\'e Nies}
\address{Department of Computer Science\\
University of Auckland\\
Private Bag 92019\\
Auckland, New Zealand}
\email{andre@cs.auckland.ac.nz}
 
\thanks{Greenberg and Nies are supported by the Marsden fund of New Zealand. The authors thank the NII Japan for a workshop at the Shonan centre 2014, where this research began. Greenberg was also supported by a Rutherford Discovery Fellowship and by a Turing Research Fellowship ``Mind, Mechanism, Mathematics'' from the Templeton Foundation.}


% 03-XX: Mathematical logic and foundations
%	03Dxx: Computability and recursion theory
%		03D25: Recursively (computably) enumerable sets and degrees
%		03D28: Other Turing degree structures
%		03D30: Other degrees and reducibilities
%		03D32: Algorithmic randomness and dimension [See also 68Q30]
% 
% 68-XX: Computer science For papers involving machine computations and programs in a specific mathematical area, see Section--04 in that area
%	68Qxx: Theory of computing
%		68Q30: Algorithmic information theory (Kolmogorov complexity, etc.) [See also 03D32]

\begin{abstract}
We study the sets that are computable from both halves of some (Martin-L\"of) random sequence, which we call \emph{$1/2$-bases}. We show that the collection of such sets forms an ideal in the Turing degrees that is generated by its c.e.\ elements. It is a proper subideal of the $K$-trivial sets. We characterise $1/2$-bases as the sets computable from both halves of Chaitin's $\Omega$, and as the sets that obey the cost function $\mathbf c(x,s) = \sqrt{\Omega_s - \Omega_x}$.

Generalising these results yields a dense hierarchy of subideals in the $K$-trivial degrees: For $k< n$, let $\+B_{k/n}$ be the collection of sets that are below any $k$ out of $n$ columns of some random sequence. As before, this is an ideal generated by its c.e.\ elements and the random sequence in the definition can always be taken to be $\Omega$. Furthermore, the corresponding cost function characterisation reveals that $\+B_{k/n}$ is independent of the particular representation of the rational $k/n$, and that $\+B_p$ is properly contained in $\+B_q$ for rational numbers $p< q$. These results are proved using a generalisation of the Loomis--Whitney inequality, which bounds the measure of an open set in terms of the measures of its projections. The generality allows us to analyse arbitrary families of orthogonal projections. As it turns out, these do not give us new subideals of the $K$-trivial sets; we can calculate from the family which $\+B_p$ it characterises.

We finish by studying the union of $\+B_p$ for $p<1$; we prove that this ideal consists of the sets that are robustly computable from some random sequence. This class was previously studied by Hirschfeldt, Jockusch, Kuyper, and Schupp~\cite{Hirschfeldt.Jockusch.ea:15}, who showed that it is a proper subclass of the $K$-trivial sets. We prove that all such sets are robustly computable from $\Omega$, and that they form a proper subideal of the sets computable from every (weakly) LR-hard random sequence. We also show that the ideal cannot be characterised by a cost function, giving the first such example of a $\Sigma^0_3$ subideal of the $K$-trivial sets.
\end{abstract}

% We characterise the ideal of sets that are computable from both halves of some (Martin-L\"of) random set $Z$ by obedience to the cost function $\mathbf c(x,s) = \sqrt{\Omega_s - \Omega_x}$. A more general version of this result yields a dense hierarchy of subideals in the $K$-trivial degrees: for $k< n$, $\+{B}_{k/n}$ is the ideal of sets that are below any $k$ out of $n$ columns of some random set $Z$. We show that this is independent of the particular representation of the rational $k/n$, and that $\+ B_p$ is properly contained in $\+ B _q$ for rationals $p< q$. We also show that $\Omega$ can be chosen as the random set without changing the ideals. The proofs use bounds on the measure of an open set in terms of the measure of its projections, in the spirit of the Loomis--Whitney inequality.

\maketitle

\tableofcontents

%%%%%%%%
%%%%%%%%
\section{Introduction}
%%%%%%%%
%%%%%%%%

If two infinite binary sequences are chosen independently at random, we would expect them not to encode common noncomputable information. This is the case, but Martin-L\"of randomness is not strong enough to guarantee this kind of independence. Martin-L\"of's notion of randomness is the most widely used in the theory of algorithmic randomness, an area that attempts to label individual binary sequences as ``random''. For a randomness notion to make sense it must hold of almost every sequence, and it must restrict the behavior of a sequence so that it has natural properties in common with almost every sequence. We want, for example, random sequences to have an equal number of $0$s and $1$s in the limit. This is a natural property shared by almost every sequence. On the other hand, we do not want to go overboard: an infinite binary sequence $X\in 2^\omega$ always has the unusual propery of being in the singleton set $\{X\}$, a property shared with no other sequence. So if we are to label sequences as ``random'', we must limit ourselves to natural properties. In practice, we must specify a countable collection of measure zero sets that ``cover'' the nonrandom sequences. In the case of Martin-L\"of randomness, we use the ``effective measure zero sets''. These are the sets of sequences for which there is an algorithm that takes as input a rational $\epsilon>0$ and, as output, generates an open cover of the set with measure less than $\epsilon$. 

Martin-L\"of randomness is strong enough to guarantee many of the properties we would want from random sequences. For example, they have an equal number of $0$s and $1$s in the limit, and in fact, satisfy the law of the iterated logarithm; when viewed as real numbers, they are points of differentiability for every computable function of bounded variation (Demuth~\cite{D:75}, see also~\cite{BMN:16}); and they satisfy Birkhoff's ergodic theorem for computable ergodic systems with respect to effectively closed sets~\cite{BDHMS:12,FGMN:12}. On the other hand, things get interesting when we look at properties of typical sequences with respect to information content. We already alluded to an example above: independent Martin-L\"of random sequences can compute the same noncomputable information. The subject of this paper is to understand exactly how complex such shared information can be. A simpler example of Martin-L\"of random sequences having an unusual property was given by Ku\v{c}era~\cite{Kucera:85} and G\'{a}cs~\cite{Gacs:EverySequence}. They showed that every sequence is computable from some random sequence. Even an incomplete random sequence may be Turing above a noncomputable, computably enumerable (c.e.) set (Ku\v{c}era~\cite{Kucera:PriorityFree}). The set of sequences that are Turing above  noncomputable c.e.\ sets has measure zero, but it is not an effective measure zero set.

This latter failure gives rise to a dual question: what kind of c.e.\ sets can be computed by an incomplete random sequence?\footnote{Henceforth in this paper ``random'' means Martin-L\"{o}f random.} The answer to this question is now known~\cite{HirschfeldtNiesStephan:UsingRandomSetsAsOracles,DM:15,CoveringProblemAnnouncement}: these are the \emph{$K$-trivial} c.e.\ sets. The notion of $K$-triviality was introduced by Solovay~\cite{Solovay:manuscript} as the antithesis of randomness: while random sequences can be characterised as those whose initial segments cannot be compressed beyond their length, the $K$-trivial sequences are those whose initial segments are maximally compressible and contain no information beyond their length. The $K$-trivial sets are computationally weak (close to being computable). There are only countably many of them (Chaitin~\cite{Chaitin}); they are all computable from c.e.\ $K$-trivial sets (Nies~\cite{Nies:LownessPropertiesAndRandomness}); and they coincide with  the \emph{low for random} sets: the sets~$A$ such that every ML-random sequence is ML-random relative to~$A$ (Nies~\cite{Nies:LownessPropertiesAndRandomness}). The robustness of $K$-triviality has been demonstrated through a series of characterisations, including several in terms of notions of weakness. One such notion is being a \emph{base for randomness}: $A$ is $K$-trivial if and only if it is computable from some sequence which is random relative to~$A$ (Hirschfeldt et al.~\cite{HirschfeldtNiesStephan:UsingRandomSetsAsOracles}). The class of~$K$-trivial sets induces an ideal in the Turing degrees.

The join $X \oplus Y$ of binary sequences $X,Y$ is the sequence  \begin{center}  $Z= X(0)Y(0)X(1)Y(1) \ldots$ \end{center} that alternates between $X$ and $Y$; we call $X$ and $Y$ the ``halves'' of $Z$. If two sequences~$X$ and~$Y$ are random relative to each other and~$A$ is computable from both~$X$ and~$Y$, then~$A$ is $K$-trivial. Two sequences~$X$ and~$Y$ are random relative to each other if and only if the pair $(X,Y)$ is random if and only if the join $X\oplus Y$ is random. For this reason we say that~$A$ is a  \emph{$1/2$-base} if there are relatively random sequences~$X$ and~$Y$, both of which compute~$A$. Note that if~$X$ and~$Y$ witness that~$A$ is a $1/2$-base, then both~$X$ and~$Y$ are random relative to~$A$. Thus, every $1/2$-base is also a base for randomness and hence~$K$-trivial. However,  not every~$K$-trivial set is a $1/2$-base (Bienvenu et al.~\cite{BGKNT:16}.) This leads to three questions:
\begin{enumerate}
	\item Are there natural witnesses for a set being a $1/2$-base? 
	\item What structure do the $1/2$-base sets induce in the Turing degrees?
	\item In what way can the $1/2$-base sets be characterised? 
\end{enumerate}
We answer these questions in this paper. To answer questions~(1) and~(2) we show:

\begin{theorem} \label{thm:half-base} \
	\begin{enumerate}
		\item Chaitin's halting probability~$\Omega$ is a universal witness for being a $1/2$-base. That is, a set~$A$ is a $1/2$-base if and only if it is computable from both halves  of~$\Omega$. 
		\item The $1/2$-base degrees form a $\Sigma^0_3$ ideal in the Turing degrees which is generated by its c.e.\ elements; the two halves of Chaitin's~$\Omega$ are an exact pair for this ideal. 
	\end{enumerate}
\end{theorem}

What would constitute an answer for the third question? In general, how can we characterise subclasses of the $K$-trivials in a coherent way? Most of the currently known 15 or so equivalent definitions of $K$-triviality cannot be modified to distinguish between single $K$-trivial sequences. $K$-triviality itself means having the {minimal possible} initial segment complexity; lowness for randomness means not derandomizing {any} random sequence---these are extreme properties, and it is not obvious how they can be adapted to yield subclasses of the $K$-trivial sets. However, there is one characterisation of the~$K$-trivial sets that is amenable to fine-tuning: characterisation by cost functions. 

We will explore cost functions in detail in \cref{sec:cost_functions}. Informally, a cost function $\cost(x,s)$ tells us how expensive it is for a computable approximation~$\seq{A_s}$ of a $\Delta^0_2$ set~$A$ to change on some value $x$ at some stage $s$; a set~$A$ obeys a cost function if the total cost accrued along some approximation is finite. Obeying a cost function tells us that a set has an approximation with ``few'' changes and so it is likely to be computationally weak. Varying the cost function allows us to quantify this notion: roughly, the higher the cost, the fewer the changes permitted and so the closer a set obeying the cost function is to being computable. The~$K$-trivial sets themselves are the sets that obey the cost function $\cost_\Omega(x,s) = \Omega_s-\Omega_x$ (Nies \cite{Nies:CalculusOfCostFunctions}, extending an argument in~\cite{Nies:LownessPropertiesAndRandomness}).

 A sufficiently more demanding cost function, i.e., one that makes changes substantially more expensive, will be obeyed by {some} but not all $K$-trivial sets. This gives us the flexibility to explore behavior within the ideal of $K$-trivial sets. The biggest success along these lines was the characterisation of the \emph{strongly jump traceable} sets (which form a proper subideal of the $K$-trivial sets) using a natural family of cost functions~\cite{GreenbergNies:benign,DGT:InherentEnumerabilityofSJT}.\footnote{Jump traceability offers another tool to distinguish $K$-trivial sets: the slower the growth of the trace bound, the smaller the class of sets that is jump traceable with that bound. A sufficiently slow bound guarantees $K$-triviality. Unfortunately, jump traceability does not seem to be well-suited to our task because it is not even known if $K$-triviality itself can be characterised using jump traceability with computable bounds.} In this paper we answer question~(3) by showing that a set is a $1/2$-base if and only if it obeys the cost function $\cost_{\Omega,1/2}(x,s) = \sqrt{\Omega_s-\Omega_x}$. 

\medskip

The results described so far are a special case. By generalising we will obtain a dense hierarchy of subideals of the $K$-trivial sets.

\begin{definition} \label{def:k/n-base}
	Let $1\le k < n$. A set~$A$ is a \emph{$k/n$-base} if there is a random $n$-tuple $(Z_1,Z_2,\dots, Z_n)$ such that~$A$ is computable from the join of any~$k$ of the sets $Z_1,Z_2,\dots, Z_n$.\footnote{see \cref{sub:capturing_the_columns_of_Omega} below for a discussion of tuples and their joins.}
\end{definition}

For the time being the notation ``$k/n$-base'' should not be taken literally as a fraction (but more akin to the expression $dy/dx$). We will justify this notation by showing that there are subideals of the $K$-trivial sets $\+{B}_p$ indexed by rational numbers $p\in (0,1)$ which respect order ($p<q$ if and only if $\+{B}_p\subset \+{B}_q$) and such that a set is a $k/n$-base if and only if it is in $\+{B}_{k/n}$. We note that a priori there is no reason to believe for example that every $1/2$-base is also a $2/4$-base, but in fact these notions are equivalent. 

\begin{theorem} \label{thm:k_n_base} 
	Let $1\le k < n$. 
	\begin{enumerate}
		\item The $n$-columns of Chaitin's~$\Omega$ (again see \cref{sub:capturing_the_columns_of_Omega}) are a universal witness for being a $k/n$-base: a set~$A$ is a $k/n$-base if and only if it is computable from the join of any~$k$ of the $n$-columns of~$\Omega$. 
		\item A set~$A$ is a $k/n$-base if and only if it obeys the cost function $\cost_{\Omega,k/n}(x,s) = (\Omega_s-\Omega_x)^{k/n}$. The collection $\+{B}_{k/n}$ of the $k/n$-base degrees is a $\Sigma^0_3$ ideal in the Turing degrees which is generated by its c.e.\ elements. 
	\end{enumerate}
\end{theorem}

The proof of \cref{thm:k_n_base} relies on the technical notion of weak obedience to cost functions which we introduce in \cref{sec:cost_functions}. We thus delay giving an outline of the proof until the end of that section. We remark that it has already been noticed by Hirschfeldt et al.~\cite{Hirschfeldt.Jockusch.ea:15} that each $k/n$-base is $K$-trivial; if $(Z_1,\dots, Z_n)$ witnesses that~$A$ is a $k/n$-base then for all $j\le n$, by van Lambalgen's theorem, $Z_j$ is random relative to $Z_{\ne j}$ (the join of the other~$Z_i$); the latter computes~$A$ and so $Z_j$ is random relative to $Z_{\ne j}\oplus A$; using van Lambalgen's theorem relative to $A$, we see that the tuple $(Z_1,\dots, Z_n)$ is $A$-random, and so $A$ is a base for randomness. 

\smallskip

The notion of a $k/n$-base is in fact still not the most general one which we can define. Given a collection $\+{F}$ of subsets of $\{1,2,\dots, n\}$ we call a set $A$ an \emph{$\+{F}$-base} if there is some random tuple $(Z_1,Z_2,\dots, Z_n)$ such that~$A$ is computable from the join $\bigoplus_{i\in F} Z_i$ for all $F\in \+{F}$. This notion is of independent interest but will actually be needed in the proof of \cref{thm:k_n_base} for ``degenerate'' $k/n$-bases. We will show that there is a rational number $\abs{\+{F}}\ge 1$ such that a set~$A$ is an $ {\+{F}}$-base if and only if it is in the ideal $\+{B}_{1/\abs{\+{F}}}$, and that in a weak sense, $\Omega$ serves as a universal witness for being an~$\+{F}$-base. 

\smallskip

The sequence of ideals $\seq {\+{B}_p}$ naturally defines two ideals: $\+{B}_{<1} = \bigcup_{p}\+{B}_p$, and $\+{B}_{>0} = \bigcap_{p}\+{B}_p$ (where again in both the union and the intersection, $p$ ranges over the rational numbers in the open interval $(0,1)$). Both ideals have interesting properties.

\begin{theorem} \label{thm:omega_base}
	$\+{B}_{>0}$ is the ideal of sets which are \emph{$1/\w$-bases}: the sets which are computable from each~$Z_n$ in an infinite random sequence $(Z_1,Z_2,\dots)$.\footnote{Equivalently, we can just require that each~$Z_n$ is random relative to the join of any finite collection of other $Z_k$'s}
\end{theorem}

The ideal $\+{B}_{<1}$ is related to coarse computability. A \emph{coarse description} of a set~$A\in 2^\w$ is a set~$B\in 2^\w$ such that the density of the symmetric difference $A\symdif B$ is 0. Say that a set~$A$ is \emph{robustly computable} from a set~$Z$ if~$A$ is computable from every coarse description of~$Z$. This notion has been investigated by Hirschfeldt et al.~\cite{Hirschfeldt.Jockusch.ea:15}, where they show that if~$A$ is robustly computable from a random set then it is $K$-trivial, in fact it is an $(n-1)/n$-base for some~$n$. They also prove that not every~$K$-trivial set is robustly computable from a random set. We show:

\begin{theorem} \label{thm:Kuyper_class}
	The following are equivalent for a set~$A$:
	\begin{enumerate}
		\item $A\in \+{B}_{<1}$ (that is, $A$~is an $(n-1)/n$-base for some~$n$). 
		\item $A$ is robustly computable from some random sequence. 
		\item $A$ is robustly computable from~$\Omega$. 
		\item There is some $\epsilon>0$ such that~$A$ is computable from all sets~$B$ such that the upper density of $B\symdif \Omega$ is below~$\epsilon$. 
	\end{enumerate}
\end{theorem}

We remark that the result mentioned above (that not every $K$-trivial set is robustly computable from a random sequence) now follows from our characterisation using cost functions; see \cref{sec:robust}. 

\smallskip

Finally we show that every set in the ideal $\+{B}_{<1}$ is computable from all LR-hard random sequences. The notion of LR-hardness (equivalent to almost everywhere domination) appears in the investigations into relative computability between random and c.e.\ sets. If~$Z$ is random but is not LR-hard then it is Oberwolfach random~\cite{BGKNT:16} and so does not compute the ``smart'' $K$-trivial sets~\cite{BGKNT:16}. It is open whether this is in fact an equivalence; it is possible that every~$K$-trivial set is computable from all LR-hard random sequences. In \cref{sec:aed}, we show that the collection of $K$-trivial sets computable from all LR-hard random sequences properly contains the ideal~$\+{B}_{<1}$.

\smallskip

We summarise our findings in \cref{fig:subclasses}.
 
\begin{figure}[ht]
	\begin{center}
		
\begin{tikzpicture}

\def\Angle{57}

\def\SJTHeight{1.15}
\def\OneOmegaHeight{2.5}
\def\KuyperHeight{5.8}
\def\LRHeight{6.9}
\def\TopHeight{7.6}

\foreach \i/\a in {0/\Angle,1/180-\Angle}
	\path [name path=Border\i] (0,0) -- (\a:10);

\path [name path=SJTLine] (-5,\SJTHeight) -- (5,\SJTHeight);
\foreach \i in {0,1}
	\path [name intersections={of={Border\i} and SJTLine, by={SJT\i}}];
\draw (SJT0) -- (SJT1);

\path [name path=OneOmegaLine] (-5,\OneOmegaHeight) -- (5,\OneOmegaHeight);
\foreach \i in {0,1}
	\path [name intersections={of={Border\i} and OneOmegaLine, by={OneOmega\i}}];
\draw (OneOmega0) -- (OneOmega1);

\path [name path=KuyperLine] (-5,\KuyperHeight) -- (5,\KuyperHeight);
\foreach \i in {0,1}
	\path [name intersections={of={Border\i} and KuyperLine, by={Kuyper\i}}];
\draw (Kuyper0) -- (Kuyper1);
\shadedraw (OneOmega0) -- (OneOmega1) -- (Kuyper1) -- (Kuyper0) -- cycle;

\path [name path=LRLine] (-5,\LRHeight) -- (5,\LRHeight);
\foreach \i in {0,1}
	\path [name intersections={of={Border\i} and LRLine, by={LR\i}}];
\draw (LR0) -- (LR1);

\path [name path=TopLine] (-5,\TopHeight) -- (5,\TopHeight);
\foreach \i in {0,1}
	\path [name intersections={of={Border\i} and TopLine, by={Top\i}}];
\draw (0,0) -- (Top0) -- (Top1) -- cycle;

\clip (Top0) -- (Top1) -- (-5,0) -- (5,0) -- cycle;

\begin{scope}
\clip (Top0) -- (Top1) -- (LR1) -- (LR0) -- cycle;
	\foreach \x in {0,...,60}
		\foreach \y in {-1,...,5}
			\node [gray!50] at ($(Top1) + 0.2*(\x,-\y+0.6)$) {\tiny ?};
\end{scope}

\node at (0,\SJTHeight*0.68) {$\textup{SJT}$};
\node at (0,0.5*\KuyperHeight + 0.5*\LRHeight) {$\le_\Tur$ every LR-hard random};
\node at (0,0.5*\LRHeight + 0.5*\TopHeight) {$K$-trivial};

\path [name path=BaseLine] (-4,0) -- (4,0);

\path [name path = OneOmegaDrop] (OneOmega0) -- ($(OneOmega0) - (0,10)$);
\path [name intersections={of=BaseLine and OneOmegaDrop, by={OneOmegaBase}}];
\draw [decorate,decoration={brace,amplitude=10pt,raise=3pt}]
(OneOmega0) -- (OneOmegaBase) node [black,midway,xshift=0.6cm,rotate=-90] {\footnotesize
$1/\omega$-base};

\path [name path = KuyperDrop] (Kuyper0) -- ($(Kuyper0) - (0,10)$);
\path [name intersections={of=BaseLine and KuyperDrop, by={KuyperBase}}];
\draw [decorate,decoration={brace,amplitude=10pt,raise=3pt}]
(Kuyper0) -- (KuyperBase) node [black,midway,xshift=0.6cm,rotate=-90] {\footnotesize
Robustly computable from a random set};

\coordinate (HalfPoint) at ($ (Kuyper1)!0.5!(OneOmega1) $);
\path [name path = HalfDrop] (HalfPoint) -- ($(HalfPoint) - (0,10)$);
\path [name intersections={of=BaseLine and HalfDrop, by={HalfBase}}];
\draw [decorate,decoration={brace,amplitude=10pt,raise=0.6cm,mirror}]
(HalfPoint) -- (HalfBase) node [black,midway,xshift=-1.1cm,rotate=90] {\footnotesize
$1/2$-base};
\draw ($ (HalfPoint) - (0.05,0) $) -- ($ (HalfPoint) - (0.6,0) $);

\coordinate (ThirdPoint) at ($ (Kuyper1)!0.66667!(OneOmega1) $);
\path [name path = ThirdDrop] (ThirdPoint) -- ($(ThirdPoint) - (0,10)$);
\path [name intersections={of=BaseLine and ThirdDrop, by={ThirdBase}}];
\draw [decorate,decoration={brace,amplitude=10pt,raise=3pt,mirror}]
(ThirdPoint) -- (ThirdBase) node [black,midway,xshift=-0.6cm,rotate=90] {\footnotesize
$1/3$-base};

\coordinate (TwoThirdsPoint) at ($ (Kuyper1)!0.33333334!(OneOmega1) $);
\path [name path = TwoThirdsDrop] (TwoThirdsPoint) -- ($(TwoThirdsPoint) - (0,10)$);
\path [name intersections={of=BaseLine and TwoThirdsDrop, by={TwoThirdsBase}}];
\draw [decorate,decoration={brace,amplitude=10pt,raise=1.2cm,mirror}]
(TwoThirdsPoint) -- (TwoThirdsBase) node [black,midway,xshift=-1.7cm,rotate=90] {\footnotesize
$2/3$-base};
\draw ($ (TwoThirdsPoint) - (0.05,0) $) -- ($ (TwoThirdsPoint) - (1.2,0) $);

\end{tikzpicture}

	\end{center}
	\caption{Subclasses of the $K$-trivial degrees.}
	\label{fig:subclasses}
\end{figure}
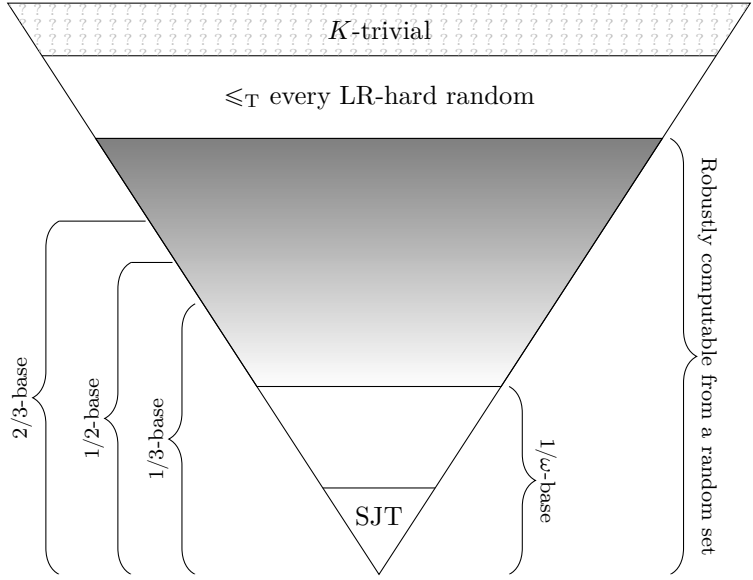

% -------

% Questions:

% (1) is $1/\infty$-base a $\Pi^0_4$-complete ideal? c.e. generated? (yes to latter)

% (2) For a real $r\in (0,1)$ say that $A$ is a $>\!r$-base if it is a $p$-base for all $p>r$. If $r$ is rational is there always a $>\!r$-base which is not an $r$-base? In general for which~$r$ are there sets which are properly $>\!r$-bases? What about $\ge\!p$-bases?

% (3) for $p\ne 1/2$, is there some natural exact pair for the $p$-bases?

% (4) Can the witness for being a $k/n$-base always be made incomplete? Is there a universal incomplete witness? Can (a) witnesses and (b) universal witnesses be characterised in some way?

% -------

%%%%%%%%
%%%%%%%%
\section{Cost functions and the corresponding test notions}
\label{sec:cost_functions}
%%%%%%%%
%%%%%%%%

Somewhat extending \cite[Section 5.3]{Nies:book}, a \emph{cost function} is a computable function
 \[ \cost\colon \NN \times \NN \ria \{x \in \RR \,:\, x \ge 0\}.
\footnote{The definition of cost functions in \cite[Section 5.3]{Nies:book} requires them to have rational values because in this case it is decidable whether a rational cost bound is satisfied. It is possible to extend the theory of cost functions to admit computable real values instead. One uses rational approximations of computable numbers in the construction of a non computable set obeying a cost function. Alternatively, for this paper we could restrict the values of cost functions to be algebraic numbers.}
 \]
We say that~$\cost$ is \emph{monotonic} if $\cost(x+1,s) \le \cost(x,s) \le \cost(x,s+1)$ for each~$x$ and~$s$; we also assume that $\cost(x,s)=0$ for all $x\ge s$. All cost functions in this paper will be monotonic without further mention. As stated above, we view $\cost(x,s)$ as the cost of changing at stage~$s$ a guess about the value $A(x)$ for some $\Delta^0_2$ set~$A$. Monotonicity means that the cost of a change increases with time and that smaller changes are more costly. 

If $\cost$ is a cost function, then we let $\limcost(x) = \lim_s \cost(x,s)$. We say that~$\cost$ satisfies the \emph{limit condition} if $\limcost(x)$ is finite for all~$x$, and $\limcost(x)\to 0$ as $x\to \infty$. As with monotonicity, all cost functions mentioned will henceforth satisfy the limit condition.

\begin{defn}[\cite{Nies:book}] \label{def:obeying_a_cost_function}
Let $\seq{A_s}$ be a computable approximation of a $\Delta^0_2$ set~$A$, and let~$\cost$ be a cost function. The \emph{total $\cost$-cost} of the approximation is 
\[ \sum_{s\in\omega} \left\{ \cost(x,s) \,: \, x \text{ is least such that } A_{s-1}(x) \ne A_{s}(x) \right\}.\]
We say that a $\Delta^0_2$ set~$A$ \emph{obeys}~$\cost$ if the total $\cost$-cost of \emph{some} computable approximation of~$A$ is finite. We write $A \models \cost$.
\end{defn}

It is not hard to show that if~$A$ is a c.e.\ set obeying a cost function~$\cost$, then there is a computable enumeration of~$A$ that witnesses this obedience; see~\cite{Nies:CalculusOfCostFunctions}.

\begin{defn} \label{def:the_p_cost_functions}
For a rational number $p \in (0,1]$ and a left-c.e.\ real~$\beta$ (equipped with an increasing approximation $\seq{\beta_s}$) we let 
\[ \cost_{\beta,p}(x,s) = (\beta_s - \beta_x)^p.\]
(As usual, we let $\cost_{\beta,p}(x,s) = 0$ if $x\ge s$.)
\end{defn}

As mentioned above, a set~$A$ is $K$-trivial if and only if it obeys the cost function $\cost_\Omega = \cost_{\Omega,1}$. Of course, if $p<q$ then $\cost_{\Omega,q}\le \cost_{\Omega,p}$, and so if $A$ obeys $\cost_{\Omega,p}$, then~$A$ obeys $\cost_{\Omega,q}$. In particular, if $p\leq 1$ and $A\models\cost_{\Omega,p}$, then $A$ is $K$-trivial.

\begin{proposition} \label{prop:obedience_of_our_cost_functions_is_downwards_closed}
	Let $p\in (0,1]$ be rational. The collection of sets that obey~$\cost_{\Omega,p}$ is downwards closed in the Turing degrees. 
\end{proposition}

Key to the proof of \cref{prop:obedience_of_our_cost_functions_is_downwards_closed} will be the fact that $K$-trivial sets do not help us approximate Chaitin's~$\Omega$ substantially better than we can with no oracle. Recall that if $B$ is $K$-trivial, it is low for Martin-L\"of randomness, so $\Omega$ is Martin-L\"of random relative to $B$.

Barmpalias and Downey~\cite[Lemma~2.5]{BD:14} proved that if $f\leq_T B$ and $\Omega$ is Martin-L\"of random relative to $B$ (i.e., $B$ is low for~$\Omega$), then there is a constant~$N$ such that $(\forall k)\; \Omega-\Omega_k < N\left(\Omega - \Omega_{f(k)}\right)$. In other words, $B$ does not allow us to speed up the approximation of $\Omega$ by more than a constant. This is exactly what we will need for the proof of \cref{prop:obedience_of_our_cost_functions_is_downwards_closed}, but we present a small improvement: in the limit, $B$ does not allow us to speed up the approximation of $\Omega$ at all.

\begin{lemma} \label{lem:lows-for-omega-approximate-Omega-badly}
	Assume that~$\Omega$ is Martin-L\"of random relative to~$B$ and that $f\colon\omega\to\omega$ is a $B$-computable function. For every $\epsilon>0$,
	\[
	(\forall^\infty k)\; \Omega-\Omega_k < (1+\epsilon)\left(\Omega - \Omega_{f(k)}\right).
	\]
	In other words,
	\[
	\liminf_{k\to\infty} \frac{\Omega - \Omega_{f(k)}}{\Omega-\Omega_k} \geq 1.
	\]
\end{lemma}
\begin{proof}
Without loss of generality, we may assume that~$f$ is strictly increasing, and that $\epsilon$ is rational. Let $n_0=0$; given $n_s$ we let $n_{s+1} = f(n_s)$. For each~$s$ let $G_s = \big(\Omega_{n_{s+1}},\Omega_{n_{s+1}}+ \frac{1}{\epsilon}\left(\Omega_{n_{s+1}}- \Omega_{n_{s-1}}\right)\big)$. Note that $\sum_{s\in\omega} \leb(G_s) \leq \frac{2}{\epsilon}\Omega$, so $\seq{G_s}$ is a $B$-Solovay test. By assumption, $\Omega$ cannot be captured by this test, so there is an $s^*$ such that if $s\geq s^*$, then $\Omega\notin G_s$. But $\Omega\geq \Omega_{n_{s+1}}$, so it must be the case that $\Omega > \Omega_{n_{s+1}}+\frac{1}{\epsilon}( \Omega_{n_{s+1}} - \Omega_{n_{s-1}})$. Rearranging:
\[
(\forall s\geq s^*)\; \epsilon\left(\Omega-\Omega_{n_{s+1}}\right) > \Omega_{n_{s+1}} - \Omega_{n_{s-1}}.
\]
Now let $k\geq n_{s^*}$. Fix $s> s^*$ such that $k\in[n_{s-1},n_{s})$. Since $f(k)< f(n_s)= n_{s+1}$,
\begin{align*}
	(1+\epsilon)\left(\Omega - \Omega_{f(k)}\right) &> (1+\epsilon)\left(\Omega - \Omega_{n_{s+1}}\right) \\
		&> \Omega - \Omega_{n_{s+1}} + \Omega_{n_{s+1}} - \Omega_{n_{s-1}} \\
		&= \Omega - \Omega_{n_{s-1}} \geq \Omega - \Omega_k. \qedhere
\end{align*}
% where we use the fact that $\Omega_{n_{s+2}}\geq \Omega_{f(n_{s+1})}$.
\end{proof}

\begin{proof}[Proof of \cref{prop:obedience_of_our_cost_functions_is_downwards_closed}]
Suppose that $A\le_\Tur B$ and~$B\models \cost_{\Omega,p}$. Let~$\Phi$ be a Turing functional such that $\Phi(B)=A$. Let~$\vphi$ be the use function for the reduction. Since $B$ is $K$-trivial and $\vphi$ is $B$-computable, we can apply \cref{lem:lows-for-omega-approximate-Omega-badly} (or Lemma~2.5 of Barmpalias and Downey~\cite{BD:14}) to get an $N>0$ such that
\[
(\forall k)\; \Omega-\Omega_k < N \left(\Omega - \Omega_{\vphi(k)}\right).
\]
The idea of the proof is to use this inequality to bound the cost of an $A$-change in terms of the cost of a corresponding $B$-change, where we take an approximation of $A$ induced by a given approximation of $B$.

Let~$\seq{B_s}$ be an approximation of~$B$ that witnesses that~$B$ obeys $\cost_{\Omega,p}$. At stage~$s$, for all $n\in \dom \Phi_s(B_s)$, let $\vphi_s(n)$ be the $\Phi_s$-use of the stage $s$ computation. We define a computable increasing sequence of stages $s(0)< s(1) < \cdots$ as follows. Let $s(0) = 0$. Given $s(i-1)$, let $s(i)>s(i-1)$ be least such that
\[
(\forall k\leq i)\; \Omega_{i+1}-\Omega_{k} < N\left(\Omega_{s(i)} - \Omega_{\vphi_{s(i)}(k)}\right).
\]
Included in this condition is the assumption that $k\in \dom \Phi_{s(i)}(B_{s(i)})$ for all $k\leq i$. The choice of $N$ guarantees that such an $s(i)$ will be found. For $i\geq 0$, let $A_i = \Phi_{s(i)}(B_{s(i)})\rest i+1$. We claim that the approximation $\seq{A_i}$ witnesses that~$A$ obeys $\cost_{\Omega,p}$.

Let $i\geq 0$ and let~$k$ be least such that $A_{i+1}(k)\neq A_i(k)$. If $k>i$, then the $\cost_{\Omega,p}$-cost accrued by the approximation $\seq{A_i}$ at stage~$i+1$ is $0$. If $k\leq i$, then the $\cost_{\Omega,p}$-cost is $(\Omega_{i+1}-\Omega_k)^p$, which is bounded by $N^p(\Omega_{s(i)}-\Omega_{\vphi_{s(i)}(k)})^p$. Let $v = \vphi_{s(i)}(k)$. The change in $A$ corresponds to a change in $B$; specifically, there must be a stage $t\in (s(i),s(i+1)]$ such that $B_t\rest{v}\ne B_{t-1}\rest{v}$. The $\cost_{\Omega,p}$-cost accrued by the approximation $\seq{B_s}$ at stage~$t$ is at least $(\Omega_t - \Omega_v)^p$, which in turn is at least $(\Omega_{s(i)}- \Omega_{\vphi_{s(i)}(k)})^p$. It follows that the total cost for~$\seq{A_j}$ is bounded by $N^p\cdot(\text{the total cost for }\seq{B_s})$, hence it is finite. 
\end{proof}

\begin{definition} \label{def:B_p}
Let $\+{B}_p$ be the collection of sets that obey~$\cost_{\Omega,p}$. 	
\end{definition}

We have seen that $\+{B}_p$ is closed downward under Turing reduction and only contains $K$-trivial sets. Nies~\cite{Nies:CalculusOfCostFunctions} proved several general results about the class of sets obeying a cost function that are helpful in understanding $\+{B}_p$. For example, $\+{B}_p$ is closed under join, which along with downward closure means that it induces an ideal in the Turing degrees. Further, every member of $\+{B}_p$ is bounded by a c.e.\ member of $\+{B}_p$, and the index set of c.e.\ members is $\Sigma^0_3$. As we have already mentioned, if $p<q$, then $\+{B}_p\subseteq \+{B}_q$. 

We will use the following, which is Theorem 3.4 of \cite{Nies:CalculusOfCostFunctions}. Here and below we write $g\le^\times h$ to mean that $g\le ch$ for some constant $c>0$.

\begin{proposition} \label{prop:calculus_quote}
	The following are equivalent for two cost functions $\cost$ and $\dcost$:
	\begin{enumerate}
		\item Every set obeying~$\cost$ also obeys $\dcost$;
		\item $\dlimcost \le^\times \limcost$. 
	\end{enumerate}
\end{proposition}

Suppose that $p<q$. Since $\limcost_{\Omega,p}$ is not bounded by any constant multiple of~$\limcost_{\Omega,q}$, the ideal $\+{B}_p$ is properly contained in $\+{B}_q$. While \cref{prop:calculus_quote} only produces a $\DII$ set in $\+{B}_q- \+{B}_p$, in our case this difference between the ideals can be witnessed by a c.e.\ set. For take $V \in \+ B_q - \+ B_p$. Since $\+ B_q$ is characterised by obedience to a cost function, there is a c.e.\ set $A \ge_T V$ in $\+ B_q$. Since $\+ B_p$ is downward closed, we have $A \not \in \+ B_p$.

\subsection{Coherent tests, and tests bounded by cost functions} % (fold)
\label{sub:coherent_tests_and_tests_bounded_by_cost_functions}

A~$\Pi^0_2$ class (i.e., an effective $G_\delta$ set) is the intersection $\bigcap_n V_n$ of a nested sequence $V_0\supseteq V_1\supseteq \cdots$ of uniformly c.e.\ open sets. Nesting ensures that the class is null if and only if $\leb(V_n)\to 0$. Such null classes characterise weak $2$-randomness. If we assume that $\leb(V_n)$ is bounded by a computable function tending to~$0$, then we have a \ML\ test. Randomness notions in between \ML\ randomness and weak $2$-randomness can be introduced by taking a suitable noncomputable witness to the fact that $\leb(V_n)\to 0$. For example, Oberwolfach randomness~\cite{BGKNT:16} can be characterised using tests satisfying $\leb(V_n)\leq \beta-\beta_n$, where $\seq{\beta_n}$ is a computable increasing sequence of approximations limiting to a left-c.e.\ real~$\beta$. In general, cost functions can be used to gauge the rate that $\leb(V_n)$ converges to~$0$. This generalises the previous example because $\cost(n,s) = \beta_s - \beta_n$ is a cost function (these are the \emph{additive} cost functions in the sense of~\cite{Nies:CalculusOfCostFunctions}).

\begin{definition}[{\cite[Def.\ 2.13]{BGKNT:16}}] \label{def:cost-bounded_test}
Let~$\cost$ be a cost function. A nested sequence $\seq{V_n}$ of uniformly c.e.\ open sets is a \emph{$\cost$-bounded test} if $\leb(V_n) \le^\times \limcost(n)$ for all~$n$.\footnote{We note that the concept of $\cost$-test in \cite{BGKNT:16} was defined without the linear constant.} 
\end{definition}

The limit condition for~$\cost$ ensures that $\leb(V_n)\to 0$, so $\bigcap_n V_n$ is indeed null. If $\seq{V_n}$ is a $\cost$-bounded test, then there is a (uniform) enumeration $\seq{V_{n,s}}$ of the sets~$V_n$ such that for all~$n$ and~$s$, $V_{n+1,s}\subseteq V_{n,s}$ and $\leb(V_{n,s})\le^\times \cost(n,s)$ (so in particular, $V_{n,n} = \emptyset$). 

\medskip

Tests bounded by additive cost functions as described above (also known as ``Auckland tests'') define the same null sets as ``Oberwolfach tests'' \cite{BGKNT:16}, which are coherent restrictions of balanced tests \cite{PaperOnBalancedTests}. The general context here is Demuth's framework for defining null sets using components that can be reset. We consider nested tests $\seq{V_n}$ where $V_n = W_{f(n)}$ for some $\Delta^0_2$ function~$f$. (Here $\seq{W_e}$ is an effective enumeration of all effectively open sets.) We require that $\leb(V_n)\le 2^{-n}$. If $\seq{f_s}$ is a computable approximation for~$f$, then $\version{V_n}{s} = W_{f_s(n)}$ is the stage~$s$ approximation of the components of the test. % (we can require that the stage~$s$ approximation is nested as well). 
The informal idea is that when building such a test we start covering some reals, but at a later stage we change our minds ($f_s(n)\ne f_{s-1}(n)$); we empty some of the components of the test and restart them. A \emph{balanced test} is such a test for which the approximation for $f(n)$ changes $O(2^n)$ times. An Oberwolfach test (a coherent balanced test) requires the changes to be coordinated across the levels of the test: if $s$ and~$t$ are successive stages at which~$f_s(n)\ne f_{s-1}(n)$ and $f_t(n)\ne f_{t-1}(n)$, then either $f_{s}(n-1)\ne f_{s-1}(n-1)$ or $f_t(n-1)\ne f_{t-1}(n-1)$. That is, every two changes in $\version{V_n}{s}$ prompt a change to $\version{V_{n-1}}{s}$. If we further assume that $V_0$ never changes, then this is equivalent to the existence of a system of (uniformly c.e.\ open) components $G_{\sigma}$ for $\sigma\in 2^{<\w}$ and a left-c.e.\ real~$\alpha\in 2^\w$ such that $\leb(G_{\sigma})\le 2^{-|\s|}$ and $V_n = G_{\alpha\rest n}$.

\begin{definition} \label{def:p-OW-tests}
	Let $p\in (0,1]$ be rational. A \emph{$p$-Oberwolfach test} consists of a left-c.e.\ binary sequence~$\alpha\in 2^\w$ and a uniformly c.e.\ open array $\seq{G_\s}_{\s\in 2^{<\w}}$ such that:
	\begin{itemize}
		\item For all $\s\in 2^{<\w}$ and $i<2$, $G_{\s\conc i}\subseteq G_{\s}$; 
		\item For all $\s\in 2^{<\w}$, $\leb(G_{\s}) \le^\times 2^{-p|\s|}$. 
	\end{itemize}
	The null set defined by the test is $\bigcap_n G_{\alpha\rest n}$. 
\end{definition}

We say that a test~$Q$ \emph{covers} a test~$P$ if the null set defined by~$P$ is a subset of the null set defined by~$Q$. The following generalises one direction of the equivalence of Auckland and Oberwolfach tests; the proof however required modification. 

\begin{proposition} \label{prop:covering_p-OW-tests_by_p-Auckland_tests}
	Let $p\in (0,1]$ be rational. Every $p$-Oberwolfach test can be covered by a $\cost_{\Omega,p}$-bounded test. 
\end{proposition}

\begin{proof}
Let $\left(\seq{G_\s},\alpha\right)$ be a $p$-Oberwolfach test. Let $\beta = \alpha+1$ and let $\beta_s = \alpha_s + (1-2^{-s})$ be the associated increasing approximation; so $\beta - \beta_s = (\alpha-\alpha_s) + 2^{-s}$. We first show that the test $\left(\seq{G_\s},\alpha\right)$ can be covered by a $\cost_{\beta,p}$-bounded test $\seq{V_n}$. We let $V_n = \bigcup_{s>n} G_{\alpha_s\rest n}$. Certainly $\bigcap_n G_{\alpha\rest n} \subseteq \bigcap_n V_n$; we need to show that $\leb(V_n)\le^\times (\beta-\beta_n)^p$. Let~$k$ be the natural number such that $2^{-k-1} \le (\alpha-\alpha_n) < 2^{-k}$. Then there are at most two strings of the form $\alpha_s\rest{k}$ for $s>n$. If $k\ge n$ then there are at most two strings of the form $\alpha_s\rest n$ for $s>n$, and so $\leb(V_n)\le^\times 2\cdot 2^{-pn}$; because $\beta-\beta_n \ge 2^{-n}$ we get $\leb(V_n)\le^\times (\beta-\beta_n)^p$. If $k<n$ then we use the fact that~$\seq{G_\s}$ is nested: in that case $V_n\subseteq \bigcup_{s>n} G_{\alpha_s\rest k}$ and so $\leb(V_n) \le^\times 2\cdot 2^{-pk}$, while $(\beta-\beta_n)^p \ge (\alpha-\alpha_n)^p\ge 1/2^p\cdot 2^{-pk}$, whence $\leb(V_n) \le^\times 2^{p+1}(\beta-\beta_n)^p$ as required. 

Next we cover~$\seq{V_n}$ by a $\cost_{\Omega,p}$-bounded test~$\seq{U_n}$. For this we use the fact that~$\Omega$ is Solovay complete; there is some increasing computable function~$f$ such that $\beta-\beta_{f(n)}\le^\times \Omega-\Omega_n$. So we let $U_n = V_{f(n)}$. 
\end{proof}

It is also the case that every $\cost_{\Omega,p}$-bounded test can be covered by a $p$-Oberwolfach test. However, we do not need this fact and do not include a proof.

% subsection coherent_tests_and_tests_bounded_by_cost_functions (end)

\subsection{Capturing the columns of \texorpdfstring{\boldmath$\Omega$}{Omega}} % (fold)
\label{sub:capturing_the_columns_of_Omega}

We will work with the computable probability space $(2^\w)^n$ for various $n<\w$, and in fact with computable probability spaces $(2^\w)^F$ where $F\subseteq \{1,2,\dots n\}$; the latter is immediately identified with $(2^\w)^{|F|}$ by using the increasing enumeration of~$F$. Elements of $(2^\w)^F$ will be denoted by uppercase Roman letters. If $Z\in (2^\w)^F$ and $i\in F$ then $Z_i$ is the $i\tth$ component of~$Z$. We identify~$n$ with $\{1,2,\dots, n\}$ so each $Z\in (2^\w)^n$ is the tuple $(Z_1,Z_2,\dots, Z_n)$.

For each $n< \w$, the computable probability space $(2^\w)^n$ is computably isomorphic to~$2^\w$ via a measure-preserving map (and so the map preserves both Turing degree and ML-randomness). There are several such maps and for most applications it does not matter which one we take. However at times it is important that we use the canonical map which distributes bits evenly: for $X\in 2^\w$ we define $j_n(X) = (X_1,X_2,\dots, X_n) \in (2^\w)^n$ by letting $X_{j+1}(k) = X(nk+j)$. The sequences~$X_1,\dots, X_n$ are called the \emph{$n$-columns} of~$X$. We also write $X = X_1\oplus X_2 \oplus \cdots \oplus X_n$. We sometimes abuse notation and write~$X$ for $j_n(X)$. However, we will denote the $n$-columns of~$\Omega$ by $\bar \Omega_1,\dots, \bar \Omega_n$, so as to not confuse them with~$\Omega_s$, the stage-$s$ approximation for~$\Omega$. 

We introduce further notation which will be useful here and later. Let $F\subseteq \{1,2,\dots, n\}$. We define the projection $\pi_F\colon (2^\w)^n\to (2^\w)^{F}$ by erasing the entries with indices outside~$F$. For clarity, for $Z\in (2^\w)^n$ we also denote $\pi_F(Z)$ by $Z_F$. Using this notation we can rephrase \cref{def:k/n-base}:

\begin{definition} \label{def:k/n_base_2}
	A set~$A$ is a \emph{$k/n$-base} if there is a random tuple $Z\in (2^\w)^n$ such that $A$ is computable from $Z_F$ for every $F\subseteq \{1,2,\dots, n\}$ of size~$k$. 
\end{definition}

The tuple~$Z$ is called a \emph{witness} for~$A$ being a $k/n$-base. The abuse of notation mentioned above results in us sometimes calling $\bigoplus_{i\le n}Z_i$ a witness as well. As promised, we will show that Chaitin's~$\Omega$ (which of course can be taken to be any left-c.e.\ random sequence) is a witness for every $k/n$-base. The following proposition is the first step toward that result.

\begin{proposition} \label{prop:covering_columns_of_Omega}
	Let $n\ge 1$ and $F\subseteq \{1,2,\dots, n\}$. Let $\bar \Omega_1,\bar \Omega_2,\dots, \bar \Omega_n$ be the $n$-columns of~$\Omega$. Then $({\bar \Omega_j})_{j\in F} = \Omega_F$ is captured by a $\cost_{\Omega,|F|/n}$-bounded test. 
\end{proposition}

\begin{proof}
	For $\s\in 2^{<\w}$, let 
$	G_\s = \left\{ X_F \,:\, X\in 2^\w \andd \s\prec X \right\}$,
where again by~$X_F$ we really mean $\pi_F(j_n(X))$. The test $\left(\seq{G_\s},\Omega\right)$ captures $\Omega_F$. If~$n$ divides~$|\s|$, then $\leb(G_\s)$ precisely equals $2^{-(|F|/n)\cdot |\s|}$ (as we specify precisely $(|F|/n)|\s|$ many bits). So in general $\leb(G_\s)$ is bounded by $2^n\cdot 2^{-|F|/n\cdot |\s|}$ and is thus an $|F|/n$-Oberwolfach test. The result follows from \cref{prop:covering_p-OW-tests_by_p-Auckland_tests}.
\end{proof}

Note that we used the fact that we are dealing with the canonical ``bits evenly distributed'' isomorphism between~$2^\w$ and $(2^\w)^n$. We cannot capture an arbitrary computable split of~$\Omega$ by a $\cost_{\Omega,1/2}$-test. As a result, it is not the case that for any computable splitting of~$\Omega$ into two parts, both parts compute every $1/2$-base.\footnote{For example, consider the 3-columns $\bar \Omega_1,\bar\Omega_2,\bar\Omega_3$ of~$\Omega$. By considering each splitting $\bar \Omega_i, \bar\Omega_j\oplus \bar \Omega_k$ (where $\{i,j,k\} = \{1,2,3\}$), we see that if for every computable splitting of~$\Omega$ into two parts, both parts compute~$A$, then $A$ is a $1/3$-base. However, not every $1/2$-base is a $1/3$-base.}  

% subsection capturing_the_columns_of_Omega (end)

\subsection{Analysis of c.e.\ \texorpdfstring{\boldmath$k/n$}{k/n} bases} % (fold)
\label{sub:analysis_of_c_e_k_n_bases}

We can now sketch the proof of \cref{thm:k_n_base} in the special case that~$A$ is c.e. The main technical result is \cref{prop:ce_k_n_bases}, which says that if~$A$ is a $k/n$-base, then~$A$ obeys the cost function $\cost_{\Omega,k/n}$. Another important fact (see for example \cite{BGKNT:16}) is that if a set~$A$ obeys a cost function~$\cost$, then it is computable from any random sequence that is captured by a $\cost$-bounded test (thus for example, any random that is not Oberwolfach random computes all $K$-trivial sets). \Cref{prop:covering_columns_of_Omega} says that any~$k$-tuple of distinct $n$-columns of~$\Omega$ can by captured by a $\cost_{\Omega,k/n}$-bounded test, and so:
\begin{itemize}
	\item each such $k$-tuple computes every $k/n$-base (since these obey $\cost_{\Omega,k/n}$);
	\item any set obeying the cost function $\cost_{\Omega,k/n}$ is a $k/n$-base. 
\end{itemize}

% subsection analysis_of_c_e_ k_n_bases (end)

\subsection{Weak obedience to cost functions} % (fold)
\label{sub:weak_obedience}

We do not know how to show directly that, in general, every $k/n$-base obeys $\cost_{\Omega,k/n}$. To overcome this we introduce a weakening of the notion of obedience.

\begin{definition} \label{def:weak_obedience}
	Let $\seq{A_s}$ be a computable approximation of a $\Delta^0_2$ set~$A$. An \emph{$n$-stage} for this approximation is a stage~$s$ at which $A_s\rest{n} = A_{s-1}\rest{n}$, $A_s(n) \ne A_{s-1}(n)$, and $A_{s}\rest{n+1} = A\rest{n+1}$. Note that there is not necessarily an $n$-stage for every~$n$, but there are $n$-stages for infinitely many~$n$. 

	Let~$\cost$ be a cost function. The \emph{weak total $\cost$-cost} of the approximation~$\seq{A_s}$ is 
	\[
		\sum \left\{ \cost(n,s) \,:\,s \text{ is the last $n$-stage} \right\}.
	\]
	The approximation $\seq{A_s}$ witnesses that~$A$ \emph{weakly obeys} the cost function~$\cost$ if the weak total $\cost$-cost of the approximation is finite. % We write $A\models_w \cost$.
\end{definition}

If~$A$ obeys~$\cost$ then it also weakly obeys it; the converse fails by the following, combined with the fact that a $K$-trivial is never Turing complete.

\begin{proposition} \label{prop:weak_does_not_imply_strong_obedience}
	There is a c.e.\ set $A\equiv_T\emptyset'$ that weakly obeys~$\cost_\Om$.
\end{proposition}

\begin{proof}
	We enumerate~$A$ as follows. For each $n\notin \emptyset'_s$, we have a marker $\gamma_s(n)\ge n$. If $n$ enters~$\emptyset'$ at stage~$s$, then we enumerate the marker $\gamma_s(n)$ into~$A_{s}$ and initialise all markers $\gamma_{s+1}(m)$ for $m>n$ to be greater than~$s$. A~$k$-stage is a ``true stage'' in the enumeration of~$A$, equivalently of~$\emptyset'$. If $s$ is a $k$-stage of the enumeration $\seq{A_s}$, then $k = \gamma_s(n)$ for some~$n$, and every number that enters~$A$ after stage~$s$ is greater than~$s$. Thus the set of intervals $I=\{ [k,s)\,:\, s\text{ is a $k$-stage of }\seq{A_s}\}$ is pairwise disjoint. The weak total $\cost_\Omega$-cost of this enumeration of $A$ is the sum of $\Omega_s-\Omega_k$, where $[k,s)$ is an interval in~$I$. Therefore, it is bounded by~$\Omega$. 
\end{proof}

Weak obedience is not very useful when we try to build our own sets. However it suffices for the following.

\begin{proposition} \label{prop:weak_obedience_and_computing}
	If~$A$ weakly obeys~$\cost$, then~$A$ is computable from any $A$-random sequence captured by a $\cost$-bounded test.
\end{proposition}

\begin{proof}
	Let~$\seq{A_s}$ be an approximation witnessing that~$A$ weakly obeys~$\cost$; let $\seq{V_n}$ be a $\cost$-bounded test. Being an $n$-stage for the approximation is recognisable by~$A$. For $n<\w$, let $G_n = \emptyset$ if there is no~$n$-stage; otherwise, let $G_n = V_{n,s}$, where~$s$ is the last $n$-stage. In other words, $G_n = \bigcup V_{n,s}$ as $s$ ranges over all~$n$-stages. Then the sequence $\seq{G_n}$ is uniformly $A$-c.e. Since $\leb(V_{n,s}) \le \cost(n,s)$, the sequence $\seq{G_n}$ is an $A$-Solovay test. 

	Suppose that $Z\in \bigcap_n V_n$ is not captured by $\seq{G_n}$; let~$r$ be the last stage at which~$Z$ enters any~$G_n$. Suppose that $Z$ enters~$V_n$ at stage $s>r$; we claim that $A_s\rest{n+1} = A\rest{n+1}$, in fact that $A_t\rest{n+1} = A_s\rest{n+1}$ for all $t\ge s$. Let~$m$ be least such that for some $t>s$, $A_t(m)\ne A_{t-1}(m)$; let~$t$ be the last such stage. Then~$t$ is an $m$-stage. Since $t>r$, $Z\notin V_{m,t}$. Since $Z\in V_{n,t}$ and the sets $\seq{V_{k,t}}$ are nested, it must be that $m>n$. Hence~$Z$ computes~$A$. 
\end{proof}

\subsubsection*{A refinement} % (fold)
\label{ssub:a_refinement}

We will require a technical refinement. Let $\+{I} = \seq{i_0,i_1,\dots}$ be a strictly increasing computable sequence. Let~$\seq{A_s}$ be a computable approximation of a $\Delta^0_2$ set~$A$. An \emph{$\+{I}$-$n$-stage} of the approximation is a stage~$s$ at which: $A_s\rest{i_n} = A_{s-1}\rest{i_n}$, $A_s\rest{i_{n+1}} \ne A_{s-1}\rest{i_{n+1}}$, and $A_{s}\rest{i_{n+1}} = A\rest{i_{n+1}}$. If~$\cost$ is a cost function, then the \emph{total $\+{I}$-weak cost} of the approximation is the sum of all $\cost(n,s)$ where~$n$ is the last $\+{I}$-$n$-stage. The approximation witnesses that~$A$ \emph{$\+{I}$-weakly obeys} $\cost$ if the total $\+{I}$-weak cost of the approximation is finite. Weak obedience is $\+I$-weak obedience for $\+I$ being the identity sequence $i_n = n$. 

The proof of \cref{prop:weak_obedience_and_computing} gives its $\+{I}$-analogue: if~$A$ $\+{I}$-weakly obeys~$\cost$ then~$A$ is computable from any $A$-random set captured by a $\cost$-bounded test. (Equivalently, we could generalise the theory of obedience and weak obedience to computably bounded elements of Baire space.)

\begin{lemma} \label{lem:I-weak_obedience_of_k_n_implies_bases}
	Let~$\+I$ be an increasing computable sequence. Suppose that~$A$ is $K$-trivial and $\+I$-weakly obeys $\cost_{\Omega,k/n}$. Then~$A$ is a $k/n$-base; in fact, the $n$-columns of~$\Omega$ witness that~$A$ is a $k/n$-base.
\end{lemma}

\begin{proof}
	Since~$A$ is low for random, any $k$-tuple of distinct $n$-columns of~$\Omega$ is $A$-random. So \cref{prop:covering_columns_of_Omega} and the $\+I$-version of \cref{prop:weak_obedience_and_computing} show that any such join computes~$A$. 
\end{proof}

% subsubsection a_refinement (end)

% subsection weak_obedience (end)

\subsection{The proof of Theorem~\ref{thm:k_n_base}} % (fold)
\label{sub:the_proof_of_cref_thm_k_n_base}

The two main results that we show later are the following. The first was mentioned above.

\begin{proposition} \label{prop:ce_k_n_bases}
	Every c.e.\ $k/n$-base obeys $\cost_{\Omega,k/n}$. 
\end{proposition}

\begin{proposition} \label{prop:general_k_n_base}
	Every $k/n$-base weakly obeys $\cost_{\Omega,k/n}$. In fact, if~$\seq{A_s}$ is any computable approximation of a $k/n$-base~$A$, then there is a sub-approximation $\seq{A_{s(n)}}$ which witnesses that~$A$ weakly obeys $\cost_{\Omega,k/n}$. 
\end{proposition}

These suffice to give a proof of \cref{thm:k_n_base}. To prove that every $k/n$-base obeys $\cost_{\Omega,k/n}$ we will show that each $k/n$-base is bounded by a c.e.\ $k/n$-base, and use \cref{prop:ce_k_n_bases,prop:obedience_of_our_cost_functions_is_downwards_closed}.
 
\begin{theorem} \label{thm:more_on_k_n_bases}
	Let $1\le k < n$. The following are equivalent for a set~$A$:
	\begin{enumerate}
		\item $A$ is a $k/n$-base;
		\item The $n$-columns of $\Omega$ witness that~$A$ is a $k/n$-base;
		\item $A$ obeys $\cost_{\Omega,k/n}$;
		\item $A$ is $K$-trivial and weakly obeys $\cost_{\Omega,k/n}$.
	\end{enumerate}
\end{theorem} 

\begin{proof}
	(2)$\then$(1) is immediate. (3)$\then$(4) holds because obedience implies weak obedience; and if $A$ obeys $\cost_{\Omega,p}$, then it obeys $\cost_\Omega$ and so is $K$-trivial. (4)$\then$(2) follows from \cref{lem:I-weak_obedience_of_k_n_implies_bases}.

	It remains to show that~(1) implies~(3). Suppose that~$A$ is a $k/n$-base. As mentioned in the introduction, we know that~$A$ is $K$-trivial~\cite{Hirschfeldt.Jockusch.ea:15}; let $\seq{\bar A_t}$ be an approximation that witnesses that~$A$ obeys~$\cost_{\Omega}$. By \cref{prop:general_k_n_base}, there is a sub-approximation $\seq{A_s} = \seq{\bar A_{t(s)}}$ of~$A$ that witnesses that~$A$ weakly obeys $\cost_{\Omega,k/n}$. This sub-approximation also witnesses that~$A$ (fully) obeys $\cost_\Omega$. As a consequence, the approximation $\seq{A_s}$ is an $\omega$-computable approximation (\cite[Fact 2.12]{Nies:CalculusOfCostFunctions}): there is a computable function~$h$ bounding the number of changes of $A_s(n)$. Using $h$, we can devise a ``reasonable'' \emph{change-set}~$C$ for the approximation $\seq{A_s}$. Let $i_n = \sum_{m\le n} h(m)$ (and let $\+{I}= \seq{i_n}$). We define the c.e.\ set~$C$ as follows: if $A_s(n)\ne A_{s-1}(n)$ and~$s$ is the $j\tth$ stage at which we saw a change in this approximation, then we enumerate $i_{n-1}+j$ into~$C_s$. Then:
	\begin{itemize}
		\item the (full) total $\cost_\Omega$-cost of the enumeration $\seq{C_s}$ is bounded by the total $\cost_\Omega$-cost of the approximation $\seq{A_s}$, and so~$C$ is $K$-trivial;
		\item the $\+{I}$-weak total $\cost_{\Omega, k/n}$-cost of the enumeration $\seq{C_s}$ equals the (normal) weak total $\cost_{\Omega, k/n}$-cost of the approximation $\seq{A_s}$, and so~$C$ $\+{I}$-weakly obeys $\cost_{\Omega, k/n}$.
	\end{itemize}
	By~\cref{lem:I-weak_obedience_of_k_n_implies_bases}, $C$ is a $k/n$-base. Since~$C$ is c.e., \cref{prop:ce_k_n_bases} says that~$C$ fully obeys $\cost_{\Omega,k/n}$. By \cref{prop:obedience_of_our_cost_functions_is_downwards_closed}, $A$ also obeys this cost function. 
\end{proof}

% subsection the_proof_of_cref_thm_k_n_base (end)

%%%%%%%%
%%%%%%%%
\section{\texorpdfstring{$1/2$}{1/2}-bases and ravenous sets}
%%%%%%%%
%%%%%%%%

We first prove \cref{prop:ce_k_n_bases,prop:general_k_n_base} for the special case $k=1$ and $n=2$. This allows us to describe the dynamics of the construction while suppressing the geometric considerations that appear in the general case. 

\subsection{Adapting the hungry sets construction} % (fold)
\label{sub:adapting_hungry_sets}

The proofs of \cref{prop:ce_k_n_bases,prop:general_k_n_base} are inspired by the ``hungry sets'' construction from \cite{HirschfeldtNiesStephan:UsingRandomSetsAsOracles}, which was used to show that every set that is a base for randomness is $K$-trivial (or low for~$K$). That argument can be transformed into a direct argument showing that every c.e.\ set~$A$ that is a base for randomness obeys the cost function $\cost_{\Omega}$. It may be instructive to sketch that argument. 

\begin{proof}[Sketch of a hungry sets construction for cost function obedience]
Let $\seq{A_s}$ be an enumeration of a c.e.\ set~$A$ (what we use is the fact that it is an increasing approximation of a left-c.e.\ real). Suppose that~$Z$ is an $A$-random that computes~$A$; let~$\Psi$ be a functional such that $\Psi(Z)=A$. 

We have a separate ``$\epsilon$-construction'' for every dyadic rational $\epsilon>0$. One of these constructions will give us the speed-up of the enumeration of~$A$ that witnesses that~$A$ obeys the cost function~$\cost_\Omega$. If an $\epsilon$-construction fails to do so, then it produces an $A$-effectively open set $U_\epsilon$ of measure at most~$\epsilon$ that contains~$Z$; so if every such construction fails, we can build an $A$-Solovay test capturing~$Z$. 

Fix $\epsilon>0$. For each string~$\tau$ we define an open set~$G_\tau$. This is the ``hungry set'' used to certify~$\tau$. It is a subset of $\Psi^{-1}[\tau] = \{X\in 2^\w\,:\, \Psi(X)\succeq \tau\}$. We require that the sets~$G_\tau$ be pairwise disjoint. The set~$G_\tau$ is satiated if its measure is precisely $\epsilon \cdot \left(\Omega_{|\tau|+1} - \Omega_{|\tau|}\right)$; this is the ``goal'' for $G_\tau$. When the set reaches its goal, we declare $\tau$ to be \emph{confirmed}. If $\tau'$ is an immediate successor of~$\tau$, we start filling the set~$G_{\tau'}$ only after~$\tau$ is confirmed; we fill $G_{\tau'}$ by clopen subsets of $\Psi^{-1}[\tau']$ disjoint from $\bigcup_{\s\preceq \tau}G_\s$.

Let $U_\epsilon = \bigcup_{n<\w} G_{A\rest n}$. Then $\leb(U_\epsilon) \le \sum \epsilon\cdot (\Omega_{n+1}-\Omega_n) \le \epsilon \cdot \Omega\le \epsilon$. If there is an~$n$ such that~$A\rest n$ is never confirmed, then $Z\in U_\epsilon$ because $\Psi^{-1}[A\rest n] \subseteq \bigcup_{m\le n} G_{A\rest{m}}$. 

Suppose that every initial segment of~$A$ is confirmed at some stage. Define an increasing sequence $s_0< s_1 < s_2 < \cdots$ of stages such that at stage~$s_n$ the string $A_{s_n}\rest{n}$ is confirmed. We claim that the total $\cost_\Omega$ cost of the approximation $\seq{A_{s_n}}$ is finite. Let $n\ge 0$ and consider the cost paid at stage~$s_{n+1}$; it is $\Omega_{n+1} - \Omega_k$ where~$k$ is the least such that $A_{s_{n+1}}(k)\ne A_{s_{n}}(k)$. Consider the set $V_n = \bigcup_{m\in (k,n]} G_{A_{s_{n}}\rest m}$. The sets $G_{A_{s_n}\rest m}$ are full and pairwise disjoint, so $\leb(V_n) = \epsilon\cdot(\Omega_{n+1}-\Omega_k)$. And further, because the approximation is left-c.e., the sets~$V_n$ are pairwise disjoint, so the total $\cost_\Omega$ cost is bounded by $(1/\epsilon)\cdot\leb(\bigcup_n V_n)$, which is of course finite. 
\end{proof}

Now consider a $1/2$-base~$A$ witnessed by a pair~$Z_1$ and~$Z_2$; say $\Psi_i(Z_i) = A$ for $i=1,2$. We attempt a similar construction. Again fixing~$\epsilon$, as before the aim is to certify strings~$\tau$ by filling hungry sets~$G_\tau$. If certification does not happen then we want to capture the pair $(Z_1,Z_2)$ by $U = \bigcup_{\tau\prec A}G_\tau$, so certification happens by letting $G_\tau\subseteq \Psi_1^{-1}[\tau]\times \Psi_2^{-1}[\tau]$. The main idea is that if a string~$\tau$ is certified with weight $\delta = \epsilon \cdot \left(\Omega_{|\tau|+1}-\Omega_{|\tau|}\right)$ and is then discovered to be wrong, then either the projection $\pi_1[G_\tau]$ to the first coordinate or its projection $\pi_2[G_\tau]$ to the second coordinate has to have measure at least $\sqrt{\delta}$: $G_\tau$ is contained in the product of the two projections. So we plan to charge the $\cost_{\Omega,1/2}$-cost of the change away from~$\tau$ to one of these projections. 

The problem is that while we can keep the hungry sets $G_\tau$ pairwise disjoint, this is not true of their projections. Consider \cref{fig:overlapping_projections}. The two rectangles represent $G_\tau$ and~$G_{\tau'}$ for an extension~$\tau'$ of~$\tau$. If we required $\pi_2[G_{\tau'}]$ to be disjoint from $\pi_2[G_\tau]$ then we would not be able to capture the pair $(Z_1,Z_2)$ in either $G_\tau$ or~$G_{\tau'}$.

\begin{figure}[h]
\centering
\begin{tikzpicture}
	[vertex/.style={circle,fill,inner sep=0pt,minimum size=2.5pt}]

\def\tauTop{3.8}
\def\tauBottom{1}
\def\tauLeft{1}
\def\tauRight{3.4}

\def\sigmaTop{5}
\def\sigmaBottom{2}
\def\sigmaRight{5.7}
\def\sigmaLeft{2.2}

\def\XAxisLength{7}
\def\YAxisLength{6}

\def\ZZ{4.4}
\def\ZO{3}

\def\NotchLength{0.14}

\coordinate (z0) at (\ZZ,0);
\coordinate (z1) at (0,\ZO);

\draw[help lines,->] (0,0) -- (0,\YAxisLength);
\draw[help lines,->] (0,0) -- (\XAxisLength,0);

\draw (\sigmaLeft,\sigmaBottom) rectangle (\sigmaRight,\sigmaTop);
\draw[fill=white] (\tauLeft,\tauBottom) rectangle (\tauRight,\tauTop);

\draw (z0) -- ($ (z0) - (0,\NotchLength) $) node [anchor=north] {$\scriptstyle Z_1$};
\draw (z1) -- ($ (z1) - (\NotchLength,0) $) node [anchor=east] {$\scriptstyle Z_2$};

\node [vertex,label=315:{$\scriptstyle (Z_1,Z_2)$}] (z) at (\ZZ,\ZO) {};
\end{tikzpicture}
\caption{Overlapping projections of hungry sets}
\label{fig:overlapping_projections}
\end{figure}
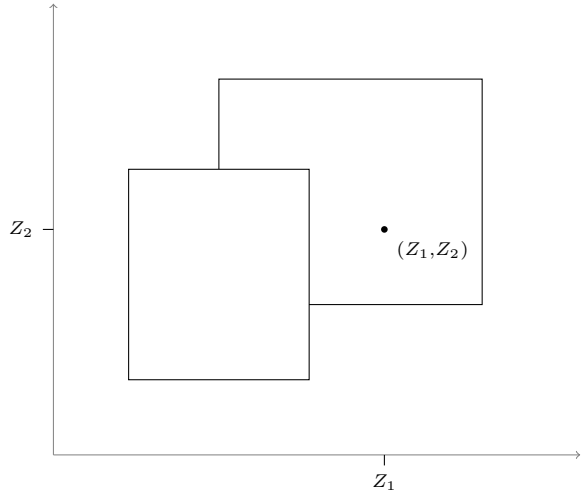

The fact that the projections of the hungry sets are not disjoint is a serious obstacle to the plan to charge the cost of changes to the measures of these projections. In the situation described, we may first discover that~$\tau'$ is incorrect and charge against the projection $\pi_1[G_{\tau'}]$ say; and later discover that~$\tau$ is incorrect and charge against $\pi_1[G_\tau]$. 

To overcome this problem, when we discover that~$\tau'$ is incorrect and want to confirm an incomparable extension $\tau''$ of~$\tau$, for our measure calculations we ignore the oracles that map to~$\tau'$. That is, from the point of view of~$\tau''$, mass has been removed from $G_\tau$ upon discovering that~$\tau'$ was incorrect, and~$\tau''$ wants to first ``refill'' $G_\tau$ before filling~$G_{\tau''}$. This may seem like a bad idea because then the telescopic sum calculation bounding the size of~$U$ is now violated. But it is not violated if instead of an $A$-Solovay test we construct a \emph{difference test}, where we are actually allowed to throw out that part of~$G_\tau$ that mapped to~$\tau'$ and not count it toward the measure of~$U$. Now however, we need to explain why the pair $(Z_1,Z_2)$ cannot be captured by such a test. For this we need the concept of measure-theoretic density.

\subsection{Density points and difference tests} % (fold)
\label{sub:density_points}

Restricting ourselves to binary density in Cantor space, for a measurable set $C\subseteq 2^\w$ and a sequence $X\in 2^\w$, the \emph{density} of~$C$ at~$X$ is defined by
\[
	\density{C}{X} = \liminf_{n\to \infty} \leb(C\!\mid\! X\rest n)
\]
where for a string $\s\in 2^{<\w}$, the conditional measure $\leb(C\!\mid\!\s)$
is $\leb(C\cap [\s])/\leb([\s])$. The Lebesgue density theorem says that for almost all $X\in C$, $\density{C}{X}=1$. A sequence~$X$ is a \emph{density one point} if $\density{C}{X}=1$ for every effectively closed ($\Pi^0_1$) set~$C$ containing~$X$. It is a \emph{positive density point} if $\density{C}{X}>0$ for every effectively closed set~$C$ containing~$X$. A random set~$X$ is a positive density point if and only if it is incomplete (Bienvenu et al.~\cite{BienvenuEtAl:DenjoyDemuthDensity}). There is an incomplete random set that is not a density one point
(Day and Miller~\cite{DM:15}). Every Oberwolfach random set is a density one point~\cite{BGKNT:16}, so every random set that is not a density one point computes every $K$-trivial set. 

The concept of dyadic density is extended to the spaces $(2^\w)^n$ (and so $(2^\w)^F$) using the standard ``evenly distributed bits'' isomorphisms~$j_n$ (see \cref{sub:capturing_the_columns_of_Omega}). It is not difficult to see that, for example, a point $(X,Y)$ in the plane $(2^\w)^2$ is a density one point if and only if for every effectively closed set $C\subseteq (2^\w)^2$ containing~$(X,Y)$, $\liminf_{n\to\infty} \leb\left(C \!\mid\! [X\rest n]\times [Y\rest n] \right)=1$. 

\smallskip

In our investigation of $1/2$-bases we will use \emph{product classes}, effectively closed subsets of the ``Cantor plane'' $(2^\w)^2$ of the form $C_1\times C_2$, where both $C_i\subseteq 2^\w$ are effectively closed. 

\begin{proposition} \label{prop:half_bases_and_positive_density_in_product_classes}
	Suppose that $(Z_1,Z_2)$ is a random pair that does not form a minimal pair (i.e., there is a noncomputable set reducible to both~$Z_1$ and~$Z_2$, so that the pair $(Z_1,Z_2)$ witnesses that some noncomputable set is a $1/2$-base). Then the pair $(Z_1,Z_2)$ has positive density in every effectively closed product class $C_1\times C_2$ containing it. 
\end{proposition}

\begin{proof}
	Let $C_1\times C_2$ be the product of two effectively closed sets; suppose that $Z_1\in C_1$, $Z_2\in C_2$, and that $\density{C_1\times C_2}{(Z_1,Z_2)}=0$. Since $\density{C_1\times C_2}{(Z_1,Z_2)}\ge \density{C_1}{Z_1}\cdot \density{C_2}{Z_2}$, either $\density{C_1}{Z_1}=0$ or $\density{C_2}{Z_2}=0$. By~\cite{BienvenuEtAl:DenjoyDemuthDensity}, either $Z_1$ or~$Z_2$ computes~$\emptyset'$, say~$Z_1$. But then $Z_2$ is 2-random, so it cannot compute any $\Delta^0_2$ set, and in particular, no $1/2$-base.
\end{proof}

\begin{remark} \label{rmk:actually_can_get_density_one}
	In fact, the assumption of \cref{prop:half_bases_and_positive_density_in_product_classes} implies that the pair $(Z_1,Z_2)$ is a density one point in product classes. If $\density{C_1\times C_2}{(Z_1,Z_2)}<1$ then either~$Z_1$ or~$Z_2$ is not a density one point, say~$Z_1$. By \cite{BienvenuEtAl:DenjoyDemuthDensity}, $Z_1$ is almost everywhere dominating, which implies that it is LR-hard \cite{Kjos.Miller.ea:11}. So again,~$Z_2$ is 2-random and forms a minimal pair with~$Z_1$. 

	When we discuss $k/n$-bases, we will need a generalisation of this fact (\cref{lem:density_in_F_product_classes-minimal_witnesses}). But \cref{prop:half_bases_and_positive_density_in_product_classes} suffices in the simple case of $1/2$-bases.
\end{remark}

The equivalence between positive density and incompleteness for random sequences passes through the notion of difference randomness. A \emph{difference test} is a sequence $\seq{P\cap G_n}$, where~$P$ is a fixed effectively closed set, $\seq{G_n}$ is uniformly effectively open and nested, and $\leb(P\cap G_n)\le 2^{-n}$; the null set defined is $P\cap \bigcap_n G_n$. Franklin and Ng~\cite{FranklinNg:Difference} showed that a random sequence~$Z$ computes~$\emptyset'$ if and only if it is captured by some difference test. Bienvenu et al.~\cite[Lemma 3.3]{BienvenuEtAl:DenjoyDemuthDensity} showed the following:

\begin{lemma} \label{lem:density_zero_and_difference_tests}
	The following are equivalent for a random sequence~$Z$ and an effectively closed set~$P$ containing~$Z$:
	\begin{enumerate}
			\item $\density{P}{Z}=0$; 
			\item Some difference test of the form $\seq{P\cap G_n}$ captures~$Z$. 
		\end{enumerate}	
\end{lemma}

Thus, \cref{lem:density_zero_and_difference_tests,prop:half_bases_and_positive_density_in_product_classes} together show that if a pair $(Z_1,Z_2)$ witnesses that some noncomputable set is a $1/2$-base, then this pair cannot be captured by a difference test whose effectively closed component is a product class. 

% subsection density_points (end)

\subsection{The c.e.\ construction} % (fold)
\label{sub:the_c_e_ construction}

We can now provide the proof of \cref{prop:ce_k_n_bases} in the case of $1/2$-bases:

\begin{proposition} \label{prop:ce_half_bases_fully_obey}
	Every c.e.\ $1/2$-base obeys $\cost_{\Omega,1/2}$. 
\end{proposition}

\begin{proof}
	Let~$A$ be a c.e.\ $1/2$-base, witnessed by the pair $(Z_1,Z_2)$. Let~$\Psi_1$ and~$\Psi_2$ be functionals such that $\Psi_i(Z_i) = A$. Let $\seq{A_s}$ be an enumeration of~$A$. For $i=1,2$, we let $P^i_s$ be the collection of oracles~$X$ such that $\Psi_{i,s}(X)$ does not lie strictly to the left of~$A_s$. We let $P^i = \bigcap_s P^i_s$, $P_s = P^1_s\times P^2_s$, and $P = \bigcap_s P_s = P^1\times P^2$. 

	\smallskip

	Let $\epsilon >0$ be a dyadic rational. We describe the $\epsilon$-construction. All sets defined henceforth depend on~$\epsilon$, and we omit mentioning this parameter.

	At stage $s<\w$, we define for all strings~$\tau$ clopen sets $G_{\tau,s} \subseteq \Psi_{1,s}^{-1}[\tau]\times \Psi_{2,s}^{-1}[\tau]$. These sets are increasing in~$s$. We refill $G_{\tau,s}$ as parts of it exit~$P_s$. The new measure could come from $G_{\s,s-1}$ where~$\s$ extends~$\tau$, so the sets~$G_{\tau,s}$ will not be pairwise disjoint. 

	We start with $G_{\tau,0} = \emptyset$ for all~$\tau$. At stage~$s$, we only add mass to $G_{\tau,s}$ for $\tau\prec A_s$, and decide whether such strings~$\tau$ are \emph{confirmed} at stage~$s$. This is done by induction on~$|\tau|$. We start with $G_{\emptystring,s}=\emptyset$; the empty string is always confirmed. Let $\tau \prec A_s$ be nonempty and suppose that we have already defined $G_{\tau',s}$ for all proper initial segments~$\tau'$ of~$\tau$, and that all these initial segments are confirmed at stage~$s$. We let $G_{\prec \tau,s} = \bigcup_{\tau'\precneq \tau} G_{\tau',s}$. 

 	We ensure that for all~$s$, 
	\[
	\leb \left( (G_{\tau,s}\setminus G_{\prec\tau,s})\cap P_{s} \right) \le \epsilon \cdot (\Omega_{|\tau|}-\Omega_{|\tau|-1}). 
	\]
	Note that this implies that $\leb(G_{\preceq\tau,s}\cap P_s) \le \epsilon\cdot\Omega_{|\tau|}$. 
	To define $G_{\tau,s}$, we add mass from $\Psi_{1,s}^{-1}[\tau]\times \Psi_{2,s}^{-1}[\tau]$ to~$G_{\tau,s-1}$, being careful to maintain the bound. Note that the bound has not already been violated because $P_{s-1}\supseteq P_s$ and $G_{\prec \tau,s-1} \subseteq G_{\prec\tau,s}$. If sufficient mass is found so that equality is obtained, then we declare~$\tau$ to be confirmed at stage~$s$ and move on to the next initial segment of~$A_s$. Otherwise, we declare~$\tau$ (and all of the longer initial segments of~$A_s$) unconfirmed at stage~$s$; we let $G_{\tau',s} = G_{\tau',s-1}$ for every $\tau'\npreceq \tau$, and move to stage~$s+1$. Observe that if~$\tau$ is confirmed at stage~$s$, then $\leb(G_{\preceq\tau,s}\cap P_s) = \epsilon\cdot\Om_{|\tau|}$.

	\smallskip

 For all~$\tau$, let $G_{\tau} = \bigcup_s G_{\tau,s}$ and $G_{\preceq \tau}= \bigcup_s G_{\preceq\tau,s} = \bigcup_{\tau'\preceq \tau} G_{\tau'}$. Then $G_{\preceq\tau}\cap P = \bigcup_s \left( G_{\preceq\tau,s}\cap P\right)$. For each~$\tau$ and~$s$, $G_{\preceq\tau,s}\cap P\subseteq G_{\preceq\tau,s}\cap P_s$ and so $\leb(G_{\preceq\tau,s}\cap P)\le \epsilon \cdot \Omega_{|\tau|}$. It follows that $\leb(G_{\preceq\tau}\cap P)\le \epsilon \cdot \Omega_{|\tau|}$. 

	Let $G= \bigcup_{\tau\prec A} G_{\tau}$. Then $G\cap P$ is the increasing union of the sets $G_{\preceq \tau}\cap P$ for $\tau\prec A$, all of measure bounded by~$\epsilon$, and so $\leb(G\cap P)\le \epsilon$. 

	Suppose that there is some~$n$ such that the string $A\rest{n}$ is confirmed during only finitely many stages; let~$n$ be the least such. Let~$t$ be sufficiently large so that $A_t\rest n\prec A$; $A_t\rest{n-1}$ is confirmed at stage~$t$; and for both $i\le 2$, $Z_i\in \Psi_{i,t}^{-1}[A\rest{n}]$. Then the fact that~$A\rest{n}$ is not confirmed at stage~$t$ implies that $G_{\preceq A\rest{n},t} \supseteq \Psi_{1,t}^{-1}[A\rest{n}] \times \Psi_{2,t}^{-1}[A\rest{n}]$ and so $(Z_1,Z_2)\in G$. 

	So there must be some~$\epsilon$ for which every initial segment of~$A$ is confirmed infinitely often. For suppose otherwise. Then the sets~$G_n$ given by the $2^{-n}$-constructions together with~$P$ form an $A$-difference test $\seq{P\cap G_n}$ which captures~$(Z_1,Z_2)$. Relativising \cref{lem:density_zero_and_difference_tests} to~$A$ and using the fact that $(Z_1,Z_2)$ is $A$-random (as $A$ is $K$-trivial), we conclude that $\density{P}{(Z_1,Z_2)} = 0$. However~$P$ is a product class; this contradicts \cref{prop:half_bases_and_positive_density_in_product_classes}.

	\smallskip

	Fix~$\epsilon$ for which every initial segment of~$A$ is confirmed infinitely often. Define an increasing sequence of stages $s_0<s_1<\cdots $ such that at stage~$s_k$ the string $A_{s_k}\rest {(k+1)}$ is confirmed. We claim that the total $\cost_{\Omega,1/2}$-cost of the enumeration $\seq{A_{s_k}}$ is finite. Let $k\ge 0$; let $x= x_k$ be the least such that $A_{s_{k+1}}(x)\ne A_{s_k}(x)$; assume that $x\le k$. The incurred cost between~$s_k$ and $s_{k+1}$ is $\sqrt{\Omega_{k+1}-\Omega_x}$. Let 
	\[ 
	D_k = \left\{ A_{s_k}\rest{(x+1)}, A_{s_k}\rest{(x+2)}, \dots, A_{s_k}\rest{(k+1)} \right\}
	\]
	and let \[ V_k = P_{s_k} \cap \bigcup_{\tau\in D_k} G_{\tau,s_k}.\]
	The fact that every string in~$D_k$ is confirmed at stage~$s_k$ implies that $\leb(V_k)= \epsilon\cdot (\Omega_{k+1}-\Omega_x)$. Hence either $\pi_1[V_k]$ or $\pi_2[V_k]$ has measure at least $\sqrt{\epsilon}\sqrt{\Omega_{k+1}-\Omega_x}$. 

	Suppose that $k<k'$. Then $\pi_1[V_k]$ and $\pi_1[V_{k'}]$ are disjoint (and the same holds for $\pi_2$). The reason is that 
	 $\pi_1[V_k]\subseteq \Psi_{1}^{-1}[A_{s_k}\rest{x_k+1}]$, which is disjoint from $P^1_{s_{k+1}}$ and so from $P^1_{s_{k'}}$, whereas $\pi_1[V_{k'}]\subseteq P^1_{s_{k'}}$. Overall, we see that the total cost paid is bounded by $2/\sqrt{\epsilon}$. 
\end{proof}

% subsection the_c_e_ construction (end)

\subsection{The general \texorpdfstring{\boldmath$1/2$}{1/2}-base construction} % (fold)
\label{sub:the_general_half_base_construction}

Weak obedience is very much weaker than full obedience. The main problem above, which drove us to use difference tests, is no longer a problem: the notion was designed so that we always charge for incompatible strings. We can therefore return to the basic hungry sets construction.

\begin{proposition} \label{prop:general_half_base_weakly_obeys}
	Every $1/2$-base weakly obeys $\cost_{\Omega,1/2}$. In fact, if $\seq{A_s}$ is a computable approximation of a $1/2$-base~$A$, then there is a computable subapproximation $\seq{A_{s_k}}$ that witnesses that~$A$ weakly obeys $\cost_{\Omega,1/2}$. 
\end{proposition}

\begin{proof}
	We simplify the proof of \cref{prop:ce_half_bases_fully_obey}. Fix a witness $(Z_1,Z_2)$ and functionals $\Psi_1$ and~$\Psi_2$ as above. 

\smallskip
	
	Fix a dyadic rational $\epsilon>0$. As before, we enumerate clopen sets $G_{\tau,s} \subseteq \Psi_{1,s}^{-1}[\tau]\times \Psi_{2,s}^{-1}[\tau]$. %We let $G_{\prec \tau,s} = \bigcup_{\s\precneq \tau} G_{\s,s}$ and similarly define $G_{\preceq\tau,s}$. 
	We ensure that for all $\tau\ne \emptyset$ and all~$s$, 
	\[
		\leb \left( G_{\preceq \tau,s}  \right) \le \epsilon \cdot (\Omega_{|\tau|}- \Omega_{|\tau|-1}). 
	\]
	In this construction, the sets $G_{\tau}$ will be pairwise disjoint. A string~$\tau$ is \emph{confirmed} at stage~$s$ if equality holds.	We start with $G_{\tau,0} = \emptyset$ for all~$\tau$. Let $s>0$ be a stage. The empty string is always confirmed and $G_{\emptystring,s}= \emptyset$. Let~$\tau$ be a string, and suppose that at stage~$s$, its immediate predecessor $\tau^-$ is already confirmed, but that~$\tau$ is not yet confirmed. We enumerate mass from $\Psi_{1,s}^{-1}[\tau]\times \Psi_{2,s}^{-1}[\tau]$ into $G_{\tau,s}$, ensuring that we do not overshoot the bound $ \epsilon \cdot (\Omega_{|\tau|}- \Omega_{|\tau|-1})$. If the bound is met, we declare that~$\tau$ is confirmed (currently and at all future stages), and go on to deal with the two immediate successors of~$\tau$. If not, then for every proper extension $\s$ of~$\tau$ we let $G_{\s,s} = G_{\s,s-1} = \emptyset$. 
	%\hl{<- Just to check: these are empty, right?}

	\smallskip

	As before we let $G_\tau = \bigcup_s G_{\tau,s}$, and $G = \bigcup_{\tau\prec A} G_\tau$; so $\leb(G) \le \epsilon$. If some initial segment~$\tau$ of~$A$ is never confirmed, then $(Z_1,Z_2)\in G$. If this holds for every $\epsilon>0$, then $(Z_1,Z_2)$ is captured by an $A$-ML test. However, since~$A$ is a $1/2$-base, it is $K$-trivial, and so low for random, and so $(Z_1,Z_2)$ is $A$-random and cannot be captured by such a test. Thus, fix some $\epsilon>0$ such that in the $\epsilon$-construction, every initial segment of~$A$ is eventually confirmed. 

	\smallskip

	Again define an increasing sequence $s_0 < s_1 < \cdots$ such that $A_{s_k}\rest{(k+1)}$ is confirmed at stage~$s_k$. We show that the weak total $\cost_{\Omega,1/2}$-cost of the enumeration $\seq{A_{s_k}}$ is finite. Let $N$ be the set of~$n$ for which the approximation $\seq{A_{s_k}}$ has an $n$-stage; for $n\in N$, let $k(n)$ be the last $n$-stage for this approximation. The weak total cost is $\sum_{n\in N} \sqrt{\Omega_{k(n)}-\Omega_n}$. 

	For $n\in N$, for brevity let $t(n) = s_{k(n)-1}$ and $\s_n = A_{t(n)}\rest{k(n)}$. Let 
	\[
		V_n = \bigcup_{m=n+1}^{k(n)} G_{\s_n\rest{m},t(n)}
	\]
	Then 
	\[
		V_n \subseteq \Psi_1^{-1}[\s_n\rest{(n+1)}]\times \Psi_2^{-1}[\s_n\rest{(n+1)}],
	\]
	and $\leb(V_n) = \epsilon\cdot (\Omega_{k(n)}-\Omega_n)$, so $\leb (\pi_i[V_n])\ge \sqrt{\epsilon} \sqrt{\Omega_{k(n)}-\Omega_{n}}$ for at least one $i\le 2$. 

	Let $n<n'$ be two elements of~$N$. Then
	% $t(n)< t(n')$ \hl{<- not clear why this is true; thankfully, I don't see where it's used}, and 
	 \[
	 A\rest{(n+1)}  = A_{t(n')}\rest{(n+1)} = A_{s_{k(n)}}\rest{(n+1)}
	 \]
	 is incomparable with $\s_n\rest{(n+1)}$, and so $\s_n\rest{(n+1)}$ and $\s_{n'}\rest{(n'+1)}$ are incomparable. For each $i\le 2$, $\pi_i[V_n]\subseteq \Psi_i^{-1}[\s_n\rest{(n+1)}]$ and the same holds for~$n'$, and so $\pi_i[V_n]$ and $\pi_i[V_{n'}]$ are disjoint. This shows that the weak total cost of the approximation $\seq{A_{s_k}}$ is bounded by $2/\sqrt{\epsilon}$. 
\end{proof}

\section{\texorpdfstring{$\+F$}{F}-bases} \label{sec:F-bases}
%%%%%%%%
%%%%%%%%

We want to adapt the proof from the previous section to the case of $k/n$-bases. In fact, it is useful and convenient to work in more generality. Recall the notation introduced in \cref{sub:capturing_the_columns_of_Omega}: if $F\subseteq \{1,2,\dots, n\}$, then $\pi_F\colon (2^\w)^n\to (2^\w)^F$ is the projection map erasing entries not indexed by elements of~$F$. For~$Z\in (2^\w)^n$ we also write $Z_F$ for~$\pi_F(Z)$.

% We introduce some notation: fix~$n\ge 1$ and let $\bar Z = (Z_1,Z_2,\dots, Z_n)\in (2^\w)^n$ and $F\subseteq \{1,2,\dots, n\}$. We let 
% \[
% 	\bar Z_F = \bigoplus \pi_F(\bar Z) = \bigoplus_{i\in F}Z_i
% \]
% where we use the join operation defined in \cref{sub:capturing_the_columns_of_Omega}.

\begin{defn} \label{def:F-base}
Let $\+F$ be a nonempty family of subsets of $\{1,\dots,n\}$. We say that~$A$ is an \emph{$\+F$-base} if there is a random tuple $Z \in (2^\w)^n$ such that $A \le_\Tur Z_F$ for all $F\in \+{F}$.
% \[
% 	(\forall F\in\+F)\; A\leq_T Z_F.
% \]
\end{defn}

We will prove that a set is an $\+F$-base if and only if it obeys $\cc_{\Om,p}$ for the appropriate choice of~$p$. The power~$p$ that corresponds to $\+F$ will be $1/\norm{F}$, where we define~$\norm{F}$ using a linear optimisation problem.

\subsection{The norm of~\texorpdfstring{\boldmath$\+F$}{F}} % (fold)
\label{sub:the_norm_of_F}

Fix a nonempty~$\+{F}\subseteq \+{P}(\{1,2,\dots, n\})$. We attempt to quantify the ``amount of disjointness'' present in $\+F$. 
\begin{itemize}
	\item An assignment $\seq{ x_F }_{F\in \+{F}}$ is a \emph{normalised weighting of the sets in $\+F$} if for all $F\in \+F$, $x_F$ is a nonnegative real number, and for all $i\in \{1,\dots, n\}$, $\sum \left\{ x_F \,:\, F\in\+F\text{ and }i\in F \right\} \le 1$. 
\end{itemize}
We let 
\[
	\norm{F} = \sup \left\{ \sum_{F\in \+F} x_F \,:\, \seq{x_F} \text{ is a normalised weighting of the sets in }\+F \right\}.
\]
For example, if $\emptyset\notin \+F$ and the sets in $\+F$ are pairwise disjoint, then $\norm{F} = |\+F|$. On the other hand, if $\bigcap\+F\neq\emptyset$ then $\norm{F}=1$. Since we have assumed that $\+F\neq\emptyset$, we always have $\norm{F}\ge 1$. If $\+F$ contains the empty set, then $\norm{F}=\infty$. Otherwise each weight in a normalised weighting $\seq{x_F}$ is bounded by~1, and so $\norm{F}\le |\+F|$. It is also the case that $\norm{F}\le n$; to see this, note that if $\emptyset\notin\+F$, then
$\sum_{F\in \+F} x_F \leq \sum_{i\le n}\sum \left\{ x_F \,:\, F\in\+F\text{ and }i\in F \right\} \leq \sum_{i\le n} 1 = n$.

\smallskip

The norm $\norm{F}$ is the solution to a linear optimisation (linear programming) problem in standard form: let $M$ be the $n\times |\+F|$-incidence matrix (the $(i,F)$-entry is 1 if $i\in F$, 0 otherwise). The problem is to maximise $\sum_{F} x_F$ under the constraints $\seq{x_F} \ge \zeroVector$ and $M\cdot \seq{x_F} \le \oneVector$ (where we think of ~$\seq{x_F}$ as a column). This problem is \emph{feasible} (the constraints are not contradictory), as is witnessed by the zero vector (or by the constant weighting $(1/|\+F|)\cdot \oneVector$, which witnesses that $\norm{F}\ge 1$). Further, if $\+F$ does not contain the empty set, then the problem is \emph{bounded} since the feasible solutions all have entries between~0 and~1. This implies that if $\+F$ does not contain the empty set, then the problem has an optimal solution, that is, there is a normalised weighting $\seq{x_F}$ such that $\sum_F x_F = \norm{F}$. Moreover, since the problem is defined with rational coefficients, $\norm{F}$ is rational and an optimal solution can be taken to consist of rational numbers; this is because $\seq{x_F}$ can be taken to be a \emph{basic feasible solution}, i.e., a vertex of the convex region defined by the constraints. For more information on linear programming, see for example~\cite{BT:97}.

\smallskip

A linear optimisation problem in standard form has a \emph{dual problem}. The dual for the problem defining $\norm{F}$ is to minimise $\sum_{i\le n} y_i$ where $\seq{y_i} \ge \zeroVector$ and $\seq{y_i}\cdot M\ge \oneVector$ (where we think of $\seq{y_i}$ as a row). Conceptually:
\begin{itemize}
	\item an $n$-tuple of nonnegative real numbers $\seq{y_i}$ is a \emph{weighting of coordinates, normalised for~$\+F$} if for all $F\in \+F$, $\sum_{i\in F} y_i\ge 1$. 
\end{itemize}
The dual problem is to minimise the sum of the weights of such a weighting. The strong duality theorem (see for example~\cite[Theorem 4.4]{BT:97}) says that as long as a linear optimisation problem has an optimal solution, then its dual problem also has an optimal solution, and the optimal values are the same. This gives us an alternate expression for $\norm{F}$ in the case that $\+F$ does not contain the empty set:
\[
	\norm{F} = \min \left\{ \sum_{i\le n} y_i \,:\, \seq{y_i} \text{ is a weighting of coordinates normalised for }\+F \right\}.
\]

\begin{example}\label{example:calculating_n_over_k}
Let $\+F$ be the collection of all $k$-element subsets of $\{1,2,\dots,n\}$. We can use both expressions for $\norm{F}$ to show that $\norm{F}=n/k$. 

In one direction, let $\seq{x_F}$ be the constant weighting of sets $x_F = \frac{n}{k}\big/\binom{n}{k}$. Each $i\le n$ is an element of precisely $\binom{n-1}{k-1} = \frac{k}{n}\binom{n}{k}$ many of the sets in~$\+F$ and so $\seq{x_F}$ is indeed a normalised weighting of the sets in~$\+F$. This weighting witnesses that $\norm{F}\ge n/k$. On the other hand, the constant weighting $y_i= {1/k}$ of coordinates is normalised for~$\+F$, so the dual expression shows that $\norm{F}\le n/k$. 
\end{example}

% subsection the_norm_of_F (end)

\subsection{\texorpdfstring{\boldmath$\+F$}{F}-bases: the easy direction} % (fold)
\label{sub:F-bases-the_easy_direction}

As promised above, we will show that a set is an~$\+F$-base if and only if it obeys $\cost_{\Omega, 1/\norm{F}}$. Actually this is not strictly true; it fails when $\norm{F}=1$. So we first discuss the extreme values. We note that for a nonempty family $\+F$ of subsets of $\{1,2,\dots, n\}$,
\begin{itemize}
	\item $\norm{F} = 1$ if and only if $\bigcap \+F \ne \emptyset$. If $\norm{F}>1$ then $\norm{F}\ge n/(n-1)$.
	\item $\norm{F} = \infty$ if and only if $\emptyset\in \+F$.
\end{itemize}
If $\norm{F}=1$, then every set is an $\+F$-base (this follows from the Ku\v{c}era--G\'{a}cs theorem), so no obedience of any cost function can be deduced. On the other hand, if $\norm{F}=\infty$, then the equivalence between being a base and obeying the cost function holds trivially: in that case being an $\+F$-base is the same as being computable, and the cost function $\cost_{\Omega,0}$ fails the limit condition and any set obeying it is computable. 

\smallskip

Henceforth we restrict ourselves to nonempty families $\+F$ satisfying $1<\norm{F}<\infty$. Note that for such a family, every $\+F$-base is an $(n-1)/n$-base. Fix~$n$ and let $\+F$ be such a family.

\begin{lemma} \label{lem:I-weak_obedience_of_F_implies_bases}
	Let~$\+I$ be an increasing computable sequence. Suppose that~$A$ is $K$-trivial and $\+I$-weakly obeys $\cost_{\Omega,1/\norm{F}}$. Then~$A$ is an $\+F$-base. 
\end{lemma}

\begin{proof}
Let $\seq{y_i}$ be an optimal solution for the dual problem defining $\norm{F}$ (an optimal normalised weighting of coordinates). We may assume that each~$y_i$ is rational. For each $i\le n$, we let $Z_i$ be a $y_i/\|\+F\|$-part of $\Omega$, with the parts chosen disjointly; since
\[
 	\sum_{i<n} y_i/\norm{F} = 1,
\]
$\Omega$ is exhausted when it is distributed between the sequences~$Z_i$ in this way. (Formally, we let~$m$ be a common denominator of the fractions $y_i/\norm{F}$ and let~$Z_i$ be the join of $m\cdot y_i/\norm{F}$-many of the $m$-columns of~$\Omega$.) If~$y_i=0$, then we let~$Z_i$ be random relative to~$\Omega$ (joined with the other~$Z_j$ for which $y_j =0$). 

Now consider $F\in\+F$. Since $\sum_{i\in F} y_i \geq 1$, $Z_F$ accounts for at least $1/\norm{F}$ of~$\Omega$, and so \cref{prop:covering_columns_of_Omega} and the $\+I$-version of \cref{prop:weak_obedience_and_computing} show that $Z_F$ computes~$A$. Therefore, $Z$ witnesses that~$A$ is an $\+F$-base. 
\end{proof}

We note that if~$y_i$ is positive for all~$i$, then the distribution of the bits of~$\Omega$ as performed in the proof of \cref{lem:I-weak_obedience_of_F_implies_bases} gives a measure-preserving, computable isomorphism $j\colon 2^\w\to (2^\w)^n$ such that $j(\Omega)$ witnesses that~$A$ is an $\+F$-base. 

\smallskip

We will prove the following generalisations of \cref{prop:ce_k_n_bases,prop:general_k_n_base}.

\begin{proposition} \label{prop:ce_F_bases}
	Every c.e.\ $\+F$-base obeys $\cost_{\Omega,1/\norm{F}}$. 
\end{proposition}

\begin{proposition} \label{prop:general_F_base}
	Every $\+F$-base weakly obeys $\cost_{\Omega,1/\norm{F}}$. In fact, if~$\seq{A_s}$ is any computable approximation of an $\+F$-base~$A$, then there is a sub-approximation $\seq{A_{s(n)}}$ that witnesses that~$A$ weakly obeys $\cost_{\Omega,1/\norm{F}}$. 
\end{proposition}

\Cref{example:calculating_n_over_k} shows that these propositions imply \cref{prop:ce_k_n_bases,prop:general_k_n_base}. We get an analogue of \cref{thm:more_on_k_n_bases}.

\begin{theorem} \label{thm:more_on_F_bases}
	The following are equivalent for a set~$A$:
	\begin{enumerate}
		\item $A$ is an $\+F$-base;
		\item $A$ obeys $\cost_{\Omega,1/\norm{F}}$;
		\item $A$ is $K$-trivial and weakly obeys $\cost_{\Omega,1/\norm{F}}$.
	\end{enumerate}
\end{theorem} 

The proof of \cref{thm:more_on_F_bases} is identical to the proof of \cref{thm:more_on_k_n_bases}, except that we use \cref{prop:ce_F_bases,prop:general_F_base} and replace \cref{lem:I-weak_obedience_of_k_n_implies_bases} by \cref{lem:I-weak_obedience_of_F_implies_bases}.

\smallskip

As discussed above, if $y_i>0$ for all~$i$ then~$\Omega$ is a universal witness for $\+F$-bases, but only in a weak sense: if~$A$ is an $\+F$-base then we can divide the digits of~$\Omega$ effectively into~$n$ many parts that together witness that~$A$ is an $\+F$-base. However, unlike the case of $k/n$-bases, we cannot require that this is the even division of bits which gives the standard isomorphism between $2^\w$ and $(2^\w)^n$. For a very simple example, let $n=3$ and $\+F = \left\{ \big\{ 1,2\big\}, \big\{ 3\big\} \right\}$. Then $\norm{F}=2$, so \cref{thm:more_on_F_bases} says that being an $\+F$-base is the same as being a $1/2$-base. A $1/2$-base may not be computable from any of the 3-columns of~$\Omega$; but if $\bar \Omega_1,\bar \Omega_2,\bar\Omega_3,\bar\Omega_4$ are the four 4-columns of~$\Omega$ then the triple $(\bar\Omega_1, \bar\Omega_2, \bar\Omega_3\oplus \bar\Omega_4)$ witnesses that every $1/2$-base is an $\+F$-base. 

% subsection F-bases_easy_direction (end)

\subsection{A generalisation of the Loomis--Whitney inequality} % (fold)
\label{sub:a_generalisation_of_the_loomis_whitney_inequality}

We need a geometric lemma for the proofs of \cref{prop:ce_F_bases,prop:general_F_base}. In the proofs of \cref{prop:ce_half_bases_fully_obey,prop:general_half_base_weakly_obeys}, we relied on the fact that the area of a $2$-dimensional set is bounded by the product of the measures of its projections into each dimension. To generalise to $\+F$-bases, we need to bound the volume of an $n$-dimensional set in terms of its projections onto the subspaces corresponding to members of $\+F$. %	Recall that for $F\subseteq\{1,\dots,n\}$ we define $\pi_F(X_1,\dots, X_n)$ by erasing the entries not indexed by elements of~$F$. 

We again fix $n\ge 1$ and a family $\+F$ of subsets of $\{1,\dots, n\}$. Recall the definition of a normalised weighting of the sets in~$\+F$ that was used for the definition of $\norm{F}$: a sequence $\seq{x_F}_{F\in \+F}$ of nonnegative real numbers such that $\sum_{\{F\,:\,i\in F\}} x_F\le 1$ for all $i\le n$.

\begin{lem}\label{lem:LW-generalized}
Suppose that $\seq{x_F}$ is a normalised weighting of the sets in~$\+F$. If $U\subseteq (2^\omega)^n$ is Borel, then 
\[
\leb(U)\le \prod_{F\in \+F} \leb(\pi_{F}[U])^{x_F}.
\]
\end{lem}

As we will see below, this result is a generalisation of the Loomis--Whitney inequality~\cite{LW:49}. It is not new; it seems to first appear in 1995 in Bollob{\'a}s and Thomason~\cite{BT:95} as an application of their ``box theorem'', where it is stated for sufficiently nice compact subsets of $\R^n$, though the restriction is not essential. To prove their main result, they (apparently independently) reprove a weaker generalisation of the Loomis--Whitney inequality that (in its discrete form) is due to Shearer in 1978. Shearer's result first appeared in~\cite{CGFS:86}, where it is proved from an inequality on entropies. This connection between entropy and combinatorics has become common; a recent paper of Madiman, Marcus, and Tetali~\cite{MMT:12} starts with:
\begin{quote}
It is well known in certain circles that there appears to exist an informal parallelism between entropy inequalities on the one hand, and set cardinality inequalities on the other.
\end{quote}
In their paper, in fact, the authors derive the discrete version of \cref{lem:LW-generalized} (i.e., \cref{clm:combinatorial} below) from one of their main results, a corresponding inequality on entropy. An earlier derivation, with entropy replaced by Kolmogorov complexity, was given by Romashchenko, Shen, and Vereshchagin~\cite{RSV:02}. \cref{clm:combinatorial} is an immediate consequence of their Theorems~1 and~2.

We provide a proof for completeness.

\begin{proof}
We will reduce the statement of the lemma to a statement of finite combinatorics. Before we do so, we make two simplifications:
\begin{enumerate}
	\item It suffices to prove the lemma for clopen sets~$U$. 
	\item We may assume that the weighting $\seq{x_F}$ is \emph{tight}: $\sum_{\{F\,:\,i\in F\}} x_F= 1$ for all $i\le n$.
\end{enumerate}

For (1), note that the desired property of the set~$U$ is certainly closed under taking countable increasing unions. If~$U$ is closed, then $U = \bigcap_k U_k$, where $\seq{U_k}$ is a decreasing sequence of clopen sets. The compactness of~$(2^\w)^n$ implies that $\pi_F[U] = \bigcap_k \pi_F[U_k]$, and so taking limits gives the desired inequality for all closed sets, and hence for all~$F_\sigma$ sets. We then use the regularity of Lebesgue measure to replace an arbitrary Borel~$U$ by an~$F_\sigma$ subset of the same measure. The projections of this subset may be smaller, but of course, this makes the inequality stronger.

\smallskip

For (2), we add ``slack sets''. Let 
	$\hat{\+F} = \+F \cup \left\{ \{i\} \,:\, i\le n \right\}.$
For every $F\in \+F$ that is not a singleton, let $\hat{x}_F = x_F$. For each $i\le n$ such that $\{i\}\notin \+F$, let $\hat{x}_{\{i\}} = 1-\sum_{\{F\,:\,i\in F\}} x_F$; if $\{i\}\in \+F$, then let $\hat{x}_{\{i\}} = x_{\{i\}}+ (1-\sum_{\{F\,:\,i\in F\}} x_F)$. Then $\seq{\hat{x}_F}$ is a tight weighting of the sets in $\hat{\+F}$. If the lemma holds for~$U$ and for~$\hat{\+F}$, then it also holds for~$U$ and for~$\+F$, as $x_F\le \hat{x}_{F}$ for all $F\in \+F$, $\leb(\pi_F[U])\le 1$, and $\leb(\pi_{\{i\}}[U])^{\hat x_{\{i\}}} \le 1$ for all~$\{i\}\in \hat{\+F}\setminus \+F$. 

\medskip

So we suppose that~$U$ is clopen and that $\seq{x_F}$ is tight. Fix $m<\w$ sufficiently large so that $2^{-m}$ is smaller than the granularity of~$U$. This means that~$U$ is the union of basic clopen sets of the form $\prod_{i\le n}[\s_i]$ where each~$\s_i$ is a binary string of length~$m$.\footnote{Note that $X_i$ does not depend on~$i$ here, but we need to index these sets since we soon deal with projections of $\prod_i X_i$.} For $i\le n$, let $X_i$ be the set~$2^m$ of binary strings of length~$m$. Define a relation $R\subseteq \prod_{i\le n} X_i$ by letting $\bar \s = (\s_1,\dots, \s_n)\in R$ if $[\bar \s] = \prod_i [\s_i]\subseteq U$ (if $\bar \s\notin R$ then $[\bar \s]\cap U = \emptyset$). The measure of~$U$ is the \emph{relative size} $d(R)$ of the relation~$R$, the number of $n$-tuples in~$R$ divided by the number of possible tuples, i.e., the size of $\prod_i X_i$; in this case the relative size is $|R|/2^{mn}$.
% (Note that this is a finite notion of density and is not to be confused with Lebesgue density.) 

Extend the notation $\pi_F$ to give the obvious function from $\prod_{i\le n} X_i$ to $\prod_{i\in F}X_i$ (erasing entries). For $F\in \+F$, the forward image $\pi_F[R]$ is the relation associated with $\pi_F[U]$: for $\bar \s\in \prod_{i\in F} X_i$, if $\bar \s\in \pi_F[R]$ then $[\bar \s]\subseteq \pi_F[U]$; otherwise $[\bar \s]$ and $\pi_F[U]$ are disjoint. Thus $\leb(\pi_F[U])$ is the density of the relation $\pi_F[R]$, namely $d(\pi_F[R]) = |\pi_F[R]|/ |\prod_{i\in F}X_i|$. The desired inequality for~$U$ follows from a   combinatorial statement:

\begin{claim} \label{clm:combinatorial}
	Let~$\seq{X_i}_{i\le n}$ be a sequence of nonempty finite sets. Let $R\subseteq \prod_{i\le n} X_i$ be a relation. Then
	 % and for~$F\in \+F$ let $S_F\subseteq \prod_{i\in F} X_i$ be a relation. Suppose that for all $\bar v\in R$ and all $F\in \+F$, $\pi_F(\bar v)\in S_F$. Then 
	\[
		d(R) \le \prod_{F\in \+F} d(\pi_F[R])^{x_F}.
	\]
\end{claim}

To visualise the situation, consider the example that defined $2/3$-bases, i.e., $n=3$ and $\+F = \big\{ \{1,2\}, \{1,3\}, \{2,3\} \big\}$. We use the optimal weighting $x_F = 1/2$ for all $F\in \+F$, which witnesses the fact that $\norm{F} = 3/2$. The measure-theoretic inequality we need bounds $\leb(U)$ by the product of the square roots of the measures of its projections onto the three orthogonal 2-dimensional planes. The corresponding combinatorial statement involves a tripartite graph with three vertex sets $X_1$, $X_2$ and $X_3$: for $F = \{i,j\}\in\+F$, $\pi_F[R]$ is a set of edges between $X_i$ and~$X_j$; the combinatorial lemma bounds the relative size of the triangle relation in the graph.

\smallskip

We prove \cref{clm:combinatorial} by induction on $\sum_{i\le n} |X_i|$ (in the example above, by induction on the number of vertices in the tripartite graph). The base case is $|X_i|=1$ for all~$i$, in which case either~$R$ is empty and has density 0, or~$R = \prod_i X_i$ and each $\pi_F[R] = \prod_{i\in F}X_i$ has density 1. For the induction step, choose some $i^*\le n$ such that $|X_{i^*}|>1$. Partition $X_{i^*}$ into two nonempty sets~$Y_0$ and~$Y_1$. For both $j=0,1$, we apply the induction hypothesis to~$R_j$, the restriction of~$R$ to those tuples whose $(i^*)\tth$ entry lies in~$Y_j$. Note that $R = R_0\cup R_1$ is a disjoint union.
We let:
\begin{itemize}
	\item for $j< 2$, $y_j = |Y_j|$ and $r_j = |R_j|$;
	\item for $F\in \+F$ and $j<2$, $s_{F,j} = |\pi_F[R_j]|$;
	\item for nonempty $F\subseteq \{1,\dots, n\}$, $z_F = \prod_{i\in F\setminus \{i^*\}} |X_i|$;
	\item and for brevity, $z^* = z_{\{1,\dots, n\} } = \prod_{i\ne i^*} |X_i|$. 
\end{itemize}
Let $\+F^* = \{ F\in \+F\,:\, i^*\in F\}$. The induction hypothesis yields, for both $j=0,1$,
\[
	\frac{r_j}{y_j z^*} \le 
	\prod_{F\in \+F^*} \left(	\frac{s_{F,j}}{y_j z_F}	\right)^{x_F} \cdot
	\prod_{F\in \+F\setminus \+F^*}	\left(	\frac{s_{F,j}}{z_F}	\right)^{x_F}.
\]
The assumption that $\sum_{F\in \+F^*}x_F = 1$ means that the occurrences of~$y_j$ cancel out. 
Let
\[
	q = z^*\cdot 
	\prod_{F\in \+F^*} \frac{1}{{z_F}^{x_F}} \cdot 
	\prod_{F\in \+F\setminus \+F^*} \left( \frac{|\pi_F[R]|}{z_F} \right)^{x_F}.
\]
For all $F\in \+F$ and $j<2$, $s_{F,j}\le |\pi_F[R]|$ and $x_F\ge 0$, so 
\begin{equation} \label{eqn:inductive_hypothesis_simplified}
		r_j \le q\cdot \prod_{F\in \+F^*} {s_{F,j}}^{x_F}.
\end{equation}
We observe that if $i^*\in F$, then $\pi_F(R) = \pi_F[R_0] \cup \pi_F[R_1]$ is a disjoint union. So to complete the induction step, we need to show that 
\[
	\frac{r_0+r_1}{(y_0+y_1) z^*} \le 
	\prod_{F\in \+F^*} \left(	\frac{s_{F,0}+ s_{F,1}}{(y_0+y_1) z_F}	\right)^{x_F} 
	\prod_{F\in \+F\setminus \+F^*}	\left(	\frac{|\pi_F[R]|}{z_F}	\right)^{x_F}.
\]
Equivalently, we need to show that 
\[
	(r_0+r_1) \le q \cdot \prod_{F\in \+F^*} (s_{F,0}+ s_{F,1})^{x_F}.
\]
This follows from \cref{eqn:inductive_hypothesis_simplified} and the inequality 
\[
	\prod_{F\in \+F^*} {s_{F,0}}^{x_F} + \prod_{F\in \+F^*} {s_{F,1}}^{x_F} \le
		\prod_{F\in \+F^*} (s_{F,0}+ s_{F,1})^{x_F},
\]
which % is a Minkowski-type inequality that
follows from the weighted arithmetic mean--geometric mean inequality using the assumption that $\sum_{F\in \+F^*}x_F = 1$.
\end{proof}

\Cref{lem:LW-generalized} can be seen as a generalisation of the Loomis--Whitney inequality~\cite{LW:49}, which bounds the measure of an $n$-dimensional set using the measures of its $(n-1)$-dimensional projections. To see the connection, we give a proof of the Loomis--Whitney result from ours.

\begin{cor}[Loomis and Whitney~\cite{LW:49}]
Let $n\ge 1$. For $j\le n$, let $\pi_j = \pi_{\{1,\dots, n\}\setminus \{j\}}$ be the projection from $[0,1]^n$ to the $(n-1)$-dimensional orthogonal subspace $x_j = 0$. If $U\subset [0,1]^n$ is Borel, then
\[
\leb(U)^{n-1}\le \prod_{j\le n} \leb(\pi_{j}[U]).
\]
\end{cor}
\begin{proof}
Apply \cref{lem:LW-generalized} with $\+F$ being the set of subsets of~$\{1,2,\dots, n\}$ of size $n-1$, and use the optimal weighting $x_F = 1/(n-1)$. 
\end{proof}

\Cref{lem:LW-generalized} gives an upper bound for the size of a set in terms of its projections. We will use it in the reverse direction, to give a lower bound for the size of one of the projections in terms of the size of the set. This can be stated in a clean, sharp form using $\norm{F}$. As usual, fix a family~$\+F$. 

\begin{lem}\label{lem:geometric}
Let $U\subseteq (2^\omega)^n$ be Borel. There is an $F\in \+F$ such that
\[
\leb(\pi_{F}[U])\ge \leb(U)^{1/\norm{F}}.
\]
Moreover, this cannot be improved: there is a $U\subseteq (2^\omega)^n$, of arbitrary measure $\le 1$, such that for all $F\in \+F$, $\leb(\pi_{F}[U])\le \leb(U)^{1/\norm{F}}$.
\end{lem}

\begin{proof}
Let $U$ be Borel and assume that $\leb(\pi_{F}[U]) < \leb(U)^{1/\norm{F}}$ for all $F\in \+F$. Let $\seq{x_F}$ be an optimal solution for the definition of $\norm{F}$ (a normalised weighting of the sets in~$\+F$ such that $\sum x_F = \norm{F}$). Apply \cref{lem:LW-generalized} using $\seq{x_F}$ to get
\[
\leb(U)\le \prod_{F\in \+F} \leb(\pi_{F}[U])^{x_F} < \prod_{F\in \+F} \leb(U)^{x_F/\norm{F}} = \leb(U),
\]
where the strict inequality follows from the fact that $\norm{F}>0$ and so $x_F>0$ for some~$F$. This is a contradiction, so there must be some $F\in \+F$ such that $\leb(\pi_{F}[U])\ge \leb(U)^{1/\norm{F}}$.

To prove sharpness, fix a measure $c$ (which must be in $[0,1]$). Let $y\in\R^n$ be an optimal solution for the dual problem defining~$\norm{F}$: a normalised weighting of the coordinates such that $\sum_i y_i = \norm{F}$. Let $U = \prod_{i\le n} \left[0,c^{y_i/\norm{F}}\right]$. Note that $\leb(U) = \prod_{i\le n} c^{y_i/\norm{F}} = c$. Fix any $F\in \+F$. Since $\sum_{i\in F} y_i \ge 1$ and $c\le 1$,
\[
\leb(\pi_{F}[U]) = \prod_{i\in F} c^{y_i/\norm{F}} = c^{\sum_{i\in F} y_i/\norm{F}} \le c^{1/\norm{F}} = \leb(U)^{1/\norm{F}}.\qedhere
\]
\end{proof}

% subsection a_generalisation_of_the_loomis_whitney_inequality (end)

\subsection{The proofs of Propositions~\ref{prop:ce_F_bases} and~\ref{prop:general_F_base}} % (fold)
\label{sub:the_proof_of_cref_prop_ce_f_bases}

% We need an elaboration on one step of the proof of \cref{prop:ce_half_bases_fully_obey}, namely 

% \begin{proposition} \label{prop:half_bases_and_positive_density_in_product_classes-relativised}
% 	Let~$A$ be a noncomputable $1/2$-base, witnessed by the pair $(Z_1,Z_2)$. Then $(Z_1,Z_2)$ has positive density in every $A$-effectively closed product class $C_1\times C_2$ containing it. 
% \end{proposition}

% \begin{proof}
% 	We give two proofs. The first relativises the argument proving the original \cref{prop:half_bases_and_positive_density_in_product_classes}. If $\density{C_1}{Z_1}=0$ then $A\oplus Z_1$ computes $A'\equiv_\Tur \emptyset'$. However $Z_1\ge_\Tur A$, so again $Z_1\ge_\Tur \emptyset'$ and~$Z_2$ is 2-random.

% 	Alternatively, we use the product version of an argument from \cite{DayMiller:Cupping}, where Day and Miller show that if~$Z$ is random, $A$ is $K$-trivial, and $P$ is an $A$-effectively closed set containing~$R$, then there is an effectively closed $Q\subseteq P$ that also contains~$Z$. The argument easily shows that this is true for product classes; that is, if $P$ is a product class, then~$Q$ can also be taken to be a product class. Hence we can replace $C_1\times C_2$ by an unrelativised product class (the density does not increase), and then appeal to \cref{prop:half_bases_and_positive_density_in_product_classes}.
% \end{proof}

We extend the proofs of \cref{prop:ce_half_bases_fully_obey,prop:general_half_base_weakly_obeys}. Say ${Z}$ witnesses that $A$ is an $\+F$-base; we fix functionals $\Psi_F$ for $F\in \+F$ such that $\Psi_F({Z}_F) = A$. The hungry sets~$G_{\tau}$ only contain tuples $ X\in (2^\w)^n$ such that $\Psi_F({X}_F)\succeq \tau$ for all $F\in \+F$. If every $\epsilon$-construction failed, then in the c.e.\ case we would capture $ Z$ by a difference test based on the effectively closed class $P = \bigcap_{F\in \+F} \pi_F^{-1}[P_F]$, where $P_F\subseteq (2^\w)^{F}$ is the class of oracles~$Y\in (2^\w)^F$ for which $\Psi_F(Y)$ does not lie to the left of~$A$. We need to show that this is impossible. As in the case of $1/2$-bases we use the concept of Lebesgue density; again by Lemma~\ref{lem:density_zero_and_difference_tests} due to~\cite{BienvenuEtAl:DenjoyDemuthDensity}, it suffices to show that the density $\density{P}{ Z}$ is positive. 

We show that \emph{if we choose the functionals cleverly}, then the density $\density{P}{ Z}$ is actually 1. We will want to show that $\density{P_F}{ Z_F}=1$ for all $F\in \+F$; this implies that $\density{\pi_F^{-1}[P_F]}{ Z}=1$, from which $\density{P}{ Z}=1$ follows. However, consider the redundant case $n=2$ and $\+F = \big\{ \{1\}, \{2\}, \{1,2\} \big\}$. In general, for any~$F$, we could choose $\Psi_F$ in such a way that $P_F$ is contained in an arbitrary effectively closed subset of $(2^\w)^F$.\footnote{If~$Q\subseteq (2^\w)^F$ is effectively closed and $\Gamma(Z_F)=A$, then we modify~$\Gamma$ so that when we discover that $\s$ drops out of~$Q$ at stage~$s$, we map all $X\succeq \s$ to $\Gamma_s(X)\conc 0^\infty$, which would lie to the left of~$A$.}
For $F = \{1,2\}$ and $Z = \Omega$, we have $Z_F = \Omega$, which is complete, and so we could choose $P_F$ so that $\density{P_F}{\Omega} = 0$. However,~$Z$ does witness that $1/2$-bases are $\+F$-bases too. The problem occurs because the reduction~$\Psi_F$ does not need to consult both parts of the oracle. %Stream of consciousness
We prove:

\begin{lem} \label{lem:density_in_F_product_classes-minimal_witnesses}
Suppose that $A$ is an $(n-1)/n$-base as witnessed by $ Z$. If $F\subseteq\{1,\dots,n\}$ is minimal such that $A\le_\Tur Z_F$, then $ Z_F$ is a density-one point for effectively closed classes in $(2^\w)^{F}$.
\end{lem}

Given the lemma, for each~$F\in \+F$ we choose some minimal $\hat F\subseteq F$ such that $A\le_\Tur Z_{\hat F}$; for~$\Psi_F$ we choose a functional that only looks at the columns indexed by elements of~$\hat F$. We then have that $P_F = Q\times (2^\w)^{F\setminus \hat F}$ where $Q\subseteq (2^\w)^{\hat F}$ is effectively closed. \Cref{lem:density_in_F_product_classes-minimal_witnesses} says that $\density{Q}{ Z_{\hat F}}=1$, from which it follows that $\density{P_F}{ Z_F}=1$, as required. As mentioned earlier, the fact that $\norm{F}>1$ means that $ Z$ witnesses that~$A$ is a $(n-1)/n$-base, so the lemma applies. 

\smallskip

To prove \cref{lem:density_in_F_product_classes-minimal_witnesses}, we use a weak van-Lambalgen-type property for Lebesgue density:

\begin{lemma} \label{lem:weak_van_Lambalgen_for_density}
Let $X_0,X_1\in 2^\w$. Suppose that $X_0$ is a density one point relative to~$X_1$, and $X_1$ is a density one point relative to~$X_0$. Then $X = (X_0,X_1)$ is a density one point.
\end{lemma}
\begin{proof} Let $ C \sub 2^\w \times 2^\w$ be an effectively closed set such that $X \in C$. For $Z\in 2^\w$, let $C_Z = \{ Y\in 2^\w\,:\, (Z,Y)\in C\}$. Let $\epsilon >0$. Since $X_1\in C_{X_0}$ and $C_{X_0}$ is effectively closed relative to $X_0$, there is an $n^*$ such that for all $m\ge n^*$, $\leb(C_{X_0} \!\mid\! X_1\rest{m})\ge 1-\epsilon$. Now let 
	\[
		P = \left\{ Z\in 2^\w\,:\, (\forall m\ge n^*)\; \leb( C_Z \!\mid\! X_1\rest{m}) \ge 1-\epsilon \right\}.
	\]
The set $P$ is effectively closed relative to~$X_1$, so there is an $n^{**}\ge n^*$ such that for all $m\ge n^{**}$, $\leb(P\!\mid\! X_0\rest{m})\ge 1-\epsilon$. So if $m\ge n^{**}$, then $\leb(C\!\mid\! X\rest{2m})\ge (1-\epsilon)^2$. 
\end{proof}

% \hl{What's known about the other direction? The hypothesis seems stronger since it implies high density in very narrow rectangles.}

\begin{proof}[Proof of \cref{lem:density_in_F_product_classes-minimal_witnesses}]
% As discussed in the proof of \cref{prop:half_bases_and_positive_density_in_product_classes-relativised}, because~$A$ is $K$-trivial, it suffices to show that $ Z_F$ is a density 1 point in (unrelativised) effectively closed classes. We can either use the Day--Miller argument from \cite{DayMiller:Cupping} (see the proof of \cref{prop:half_bases_and_positive_density_in_product_classes-relativised}), or relativise the following argument to~$A$ where appropriate. 
%
By permuting, we may assume that $F= \{1,2,\dots, k\}$ for some $k<n$. Let $W = (Z_{k+1},\dots, Z_n)$. For $i\le k$ let $Y_i= (Z_1,\dots, Z_{i-1},Z_{i+1},\dots, Z_k)$. 

First we see that for all $i\le k$, $Z_i$ is a density 1 point relative to $Y_i$. Suppose not. By \cite{BienvenuEtAl:DenjoyDemuthDensity}, a random set~$X$ that is not a density 1 point is \emph{LR-hard}: $\emptyset'\le_{LR} X$, which means that every $X$-random set is 2-random. Relativising, we see that $(Z_i, Y_i)\ge_{LR} Y_i'$. Now~$W$ is random relative to $ Z_F = (Z_i,Y_i)$ and so is $2$-random relative to~$Y_i$. Every weakly 2-random set forms a minimal pair with $\emptyset'$. Relativising to~$Y_i$, every set that is computable from both~$(W,Y_i)$ and $Y_i'$ is $Y_i$-computable. Since $(W,Y_i)$ consists of $n-1$ many columns of $ Z$, it computes~$A$. Also $A$ is $\Delta^0_2$, so it  certainly is $Y_i'$-computable. Hence $A\le_\Tur Y_i$, contradicting the minimality of~$F$. 

Now by induction on $i\le k$, we see that $(Z_1,\dots, Z_i)$ is a density 1 point relative to $(Z_{i+1},\dots, Z_k)$. This is already established for $i=1$. Let $i>1$ and suppose that $(Z_1,\dots, Z_{i-1})$ is a density 1-point relative to $(Z_i,\dots, Z_k)$. We use \cref{lem:weak_van_Lambalgen_for_density} relativised to $(Z_{i+1},\dots, Z_k)$ (and the fact that $Z_i$ is a density 1 point relative to~$Y_i$) to obtain the desired result. 
\end{proof}

For the benefit of the reader, we sketch the proof of \cref{prop:ce_F_bases} (\cref{prop:general_F_base} is again easier). 

\begin{proof}[Sketch of the proof of Proposition~\ref{prop:ce_F_bases}]
	We explain how to modify the proof of \cref{prop:ce_half_bases_fully_obey}. Let~$A$ be a c.e.\ $\+F$-base, witnessed by the tuple $Z=(Z_1,Z_2,\dots, Z_n)$. For $F\in \+F$ wisely choose a functional $\Psi_F$ such that $\Psi_F(Z_F) = A$, as discussed after the statement of \cref{lem:density_in_F_product_classes-minimal_witnesses}; it only looks at oracles for a minimal $\hat F\subseteq F$. For each $F\in \+F$ and $s<\w$ we let $P_{F,s}$ be the set of $X\in (2^\w)^F$ such that $\Psi_{F,s}(X)$ does not lie to the left of $A_s$. We let $P_s = \bigcap_{F\in \+F} \pi^{-1}_F [P_{F,s}]$. 

	Again we fix a dyadic rational $\epsilon>0$, and enumerate clopen sets $G_{\tau,s}\subseteq (2^\w)^n$, with $\Psi_{F,s}(\pi_F(X))\succeq \tau$ for all $F\in \+F$ and $X\in G_{\tau,s}$. The goal $\epsilon \cdot (\Omega_{|\tau|-\Omega_{|\tau|-1}})$ for $(G_{\tau}\setminus G_{\tau^-})\cap P$ is the same, as well as the confirmation process and the instructions of how to increase each $G_{\tau,s}$. 

	The definitions, at the end of the construction, of~$P$ and~$G$ are the same, as well as the argument that $\leb(G\cap P)\le \epsilon$. Similar also is the argument that if some initial segment of~$A$ is confirmed at only finitely many stages then $Z\in G\cap P$. If this happens for every~$\epsilon$, then $Z$ is captured by the $A$-difference test $\seq{P\cap G_\epsilon}$. As before it follows that $\density{P}{Z}= 0$. As described above, for each $F\in \+F$, since $Z_{\hat F}\in P_{\hat F}$, $\density{P_{\hat F}}{Z_{\hat F}} = 1$ (\cref{lem:density_in_F_product_classes-minimal_witnesses}), and so $\density{P_F}{Z_F} = 1$, so $\density{\pi^{-1}_F[P_F]}{Z} = 1$, so $\density{P}{Z}=1$. 

	We again choose~$\epsilon$ such that in the $\epsilon$-construction, every initial segment of~$A$ is confirmed infinitely often. As above, we define the increasing computable sequence~$\seq{s_k}$ so that $A_{s(k)}\rest{(k+1)}$ is confirmed at stage~$s_k$. We also define~$V_k$ exactly as above. Again the fact that every string $A_{s_k}\rest{x+1}, \dots, A_{s_k}\rest{(k+1)}$ is confirmed implies that $\leb(V_k)= \epsilon\cdot (\Omega_{k+1}-\Omega_x)$. For every $F\in \+F$ and every~$k$, $\pi_F[V_k]\subseteq \Psi_F^{-1}[A_{s_k}\rest{(x_k+1)}]$, which is disjoint from $P_{F,s_{k'}}$ for all $k'>k$, whereas $\pi_F[V_k]\subseteq P_{F,s_k}$ for all~$k$. Hence for $k<k'$ we get $\pi_F[V_k]\cap \pi_F[V_{k'}] = \emptyset$ for all $F\in \+F$. Finally, \cref{lem:geometric} shows that for every~$k$ there is some $F\in \+F$ such that 
	 \[
	 \leb(\pi_F[V_k])\ge (\epsilon\cdot (\Omega_{k+1}-\Omega_x))^{1/\norm{F}}. 
	 \]
	This shows that the total $\cost_{\Omega, 1/\norm{F}}$-cost of this enumeration is bounded by $|\+F|/ \epsilon^{1/\norm{F}}$. 
\end{proof}

% subsection the_proof_of_cref_prop_ce_f_bases (end)

%%%%%%%%
%%%%%%%%
\section{Consequences of the characterisation of \texorpdfstring{$\+F$}{F}-bases}
%%%%%%%%
%%%%%%%%

The generality of the development in the previous section allows us to prove a number of interesting results. The first was already mentioned, namely the characterisation of $k/n$-bases: \cref{example:calculating_n_over_k} shows that \cref{prop:ce_F_bases,prop:general_F_base} imply \cref{prop:ce_k_n_bases,prop:general_k_n_base} and so complete our proof of \cref{thm:more_on_k_n_bases,thm:k_n_base}.

%%%%%%%%
\subsection{Cyclic \texorpdfstring{\boldmath $k/n$}{k/n}-bases}
%%%%%%%%

Note that $\binom{n}{k}$ can be quite large compared to $n$, especially if $k\approx n/2$. This makes the definition of a $k/n$-base look very demanding, as  it requires a set to be computable from a large number of different random tuples. It turns out that we can get away with a weaker hypothesis. Fix natural numbers $0<k<n$. For each $i\le n$, let $F_i = \{i, i+1 \pmod{n},\dots, i+k-1 \pmod{n}\}$. Let $\+F = \{F_i\}_{i<n}$. We call a set~$A$ a \emph{cyclic $k/n$-base} if it an $\+F$-base. Note that this definition only requires~$A$ to be computable from~$n$ distinct tuples. And yet, it is enough to capture that~$A$ is a $k/n$-base.

\begin{prop}
A set is a cyclic $k/n$-base if and only if it is a $k/n$-base.
\end{prop}
\begin{proof}
Clearly, every $k/n$-base is a cyclic $k/n$-base. So assume that $A$ is a cyclic $k/n$-base. Let $\+F$ be as above. So $A$ is an $\+F$-base. We show that $\norm{F} = n/k$; \cref{thm:more_on_F_bases,thm:more_on_k_n_bases} imply that $A$ is a $k/n$-base.

Again we use the duality in the definition of~$\norm{F}$. To bound the norm from below consider the constant weighting $x_F = 1/k$ for all $F\in \+F$. This is normalised since every $i\le n$ is an element of precisely $k$ many sets in~$\+F$. Hence $\norm{F}\ge n/k$. From above, consider the weighting $y_i = 1/k$; each set in~$\+F$ has size~$k$ and so $\seq{y_i}$ is normalised for~$\+F$. Hence $\norm{F}\le n/k$. 
\end{proof}

Since every $k/n$-base is witnessed by the $n$-columns of~$\Omega$, so is every cyclic $k/n$-base.

%%%%%%%%
\subsection{Degenerate \texorpdfstring{\boldmath $k/n$}{k/n}-bases}
%%%%%%%%

Assume that $1<k<n$. We call $A$ a \emph{degenerate} $k/n$-base if there is a random tuple $ Z$ that witnesses that $A$ is a $k/n$-base but this is not tight: there is some $G\subseteq\{1,\dots,n\}$ such that $|G|<k$ and $A\leq_T Z_G$. We show that degenerate $k/n$-bases must obey cost functions that are stronger than $\cc_{\Om,k/n}$.

\begin{prop}\label{prop:degenerate}
Let $p=\max\big\{\frac{k}{n+1},\frac{k-1}{n-1}\big\}$. A set $A$ is a degenerate $k/n$-base if and only if it is a $p$-base.
\end{prop}
\begin{proof} %\hl{Could this be in linear programming books?}
Let $\+F$ consist of all $k$-element subsets of $\{1,\dots,n\}$ along with $G = \{1,\dots,k-1\}$. Note that a set is a degenerate $k/n$-base if and only if it is an $\+F$-base, so by \cref{thm:more_on_F_bases}, all we have to do is prove that $\norm{F}=1/p$. Let $M$ be the matrix from the definition of $\|\+F\|$. There are two cases.

\emph{Case 1:} $2k-1\leq n$. In this case, it is easy to check that $p = k/(n+1)$. Consider the following vector $x\in\R^{|\+F|}$. We will indicate the coordinate of $x$ corresponding to $F\in\+F$ by $x_F$. Let $x_G = 1$. If $F$ is a $k$-element subsets of $\{k-1,\dots,n-1\}$, let $x_F = \frac{n-k+1}{k\binom{n-k+1}{k}}$ (note that such sets exists because $n-k+1\geq k$). Let the other coordinates of $x$ be $0$. We claim that $Mx = \1$. If $i\in \{0,\dots,k-2\}$, then $Mx(i) = \sum \{x_F\,:\, F\in\+F\text{ and }i\in F\} = x_G = 1$. If $i\in \{k-1,\dots,n-1\}$, then $i$ is in a fraction of $k/(n-k+1)$ of the $k$-element subsets of $\{k-1,\dots,n-1\}$. There are $\binom{n-k+1}{k}$ such sets $F$, each with $x_F = \frac{n-k+1}{k\binom{n-k+1}{k}}$, so $Mx(i) = \sum \{x_F\,:\, F\in\+F\text{ and }i\in F\} = 1$. This proves that $Mx = \1$, hence
\[
\|\+F\| \geq \sum_{F\in\+F} x_F = 1 + \binom{n-k+1}{k}\frac{n-k+1}{k\binom{n-k+1}{k}} = \frac{n+1}{k} = 1/p.
\]

Next consider the following vector $y\in R^n$. For $i\in \{0,\dots,k-2\}$, let $y_i = 1/(k-1)$. For $i\in\{k-1,\dots,n-1\}$, let $y_i = 1/k$. We claim that $M^Ty \geq \1$. Again, we use elements of $\+F$ to index the corresponding dimensions. Note that $M^Ty(G) = \sum_{i\in G} y_i = (k-1)\frac{1}{k-1} = 1$. For any other $F\in\+F$ we have $M^Ty(F) = \sum_{i\in F} y_i \geq \sum_{i\in F} \frac{1}{k} = k\frac{1}{k} = 1$. This proves that $M^Ty \geq \1$, hence
\[
\|\+F\| \leq \sum_{i<n} y_i = (k-1)\frac{1}{k-1} + (n-k+1)\frac{1}{k} = \frac{n+1}{k} = 1/p.
\]
Therefore, $\|\+F\| = 1/p$.

\emph{Case 2:} $2k-1 > n$. In this case, it is easy to check that $p = (k-1)/(n-1)$. Consider the following vector $x\in\R^{|\+F|}$. Let $x_G = (n-k)/(k-1)$. If $F$ is a $k$-element subset of $\{0,\dots,n-1\}$ that contains $\{k-1,\dots,n-1\}$, let $x_F = 1/\binom{k-1}{2k-1-n}$ (note that such sets exist because $k>n-k+1$). Let the other coordinates of $x$ be $0$. We claim that $Mx = \1$. If $i\in \{0,\dots,k-2\}$, then $i$ is in $G$ and in a fraction of $(2k-1-n)/(k-1)$ of the $k$-element subsets of $\{0,\dots,n-1\}$ that contain $\{k-1,\dots,n-1\}$. There are $\binom{k-1}{2k-1-n}$ such sets. Therefore, $Mx(i) = \sum \{x_F\,:\, F\in\+F\text{ and }i\in F\} = (n-k)/(k-1) + (2k-1-n)/(k-1) = 1$. On the other hand, if $i\in \{k-1,\dots,n-1\}$, then $i$ is not in $G$ but is in every $k$-element subset of $\{0,\dots,n-1\}$ that contains $\{k-1,\dots,n-1\}$. So $Mx(i) = \sum \{x_F\,:\, F\in\+F\text{ and }i\in F\} = 1$. This proves that $Mx = \1$, hence
\[
\|\+F\| \geq \sum_{F\in\+F} x_F = \frac{n-k}{k-1} + \binom{k-1}{2k-1-n}\frac{1}{\binom{k-1}{2k-1-n}} = \frac{n-1}{k-1} = 1/p.
\]

Next consider the following vector $y\in R^n$. For $i\in \{0,\dots,k-2\}$, let $y_i = 1/(k-1)$. For $i\in\{k-1,\dots,n-1\}$, let $y_i = \frac{n-k}{(n-k+1)(k-1)}$. We claim that $M^Ty \geq \1$. As in Case~1, $M^Ty(G) = \sum_{i\in G} y_i = (k-1)\frac{1}{k-1} = 1$. Consider any other $F\in\+F$. At least $k-(n-k+1) = 2k-n-1$ coordinates in $F$ are from $\{0,\dots,k-2\}$. Therefore, $M^Ty(F) = \sum_{i\in F} y_i \geq (2k-n-1)\frac{1}{k-1} + (n-k+1)\frac{n-k}{(n-k+1)(k-1)} = 1$. This proves that $M^Ty \geq \1$, hence
\[
\|\+F\| \leq \sum_{i<n} y_i = (k-1)\frac{1}{k-1} + (n-k+1)\frac{n-k}{(n-k+1)(k-1)} = \frac{n-1}{k-1} = 1/p.
\]
Therefore, $\|\+F\| = 1/p$.
\end{proof}

\begin{cor}
There is a  (c.e.)\ $k/n$-base that is not a degenerate $k/n$-base.
\end{cor}
\begin{proof}
It is easy to see that both $k/(n+1)$ and $(k-1)/(n-1)$ are less than $k/n$. Therefore, by \cref{prop:calculus_quote}, there is a c.e.\ set $A$ that obeys $\cc_{\Om,k/n}$ but not $\cc_{\Om,p}$ for $p=\max\big\{\frac{k}{n+1},\frac{k-1}{n-1}\big\}$. By Proposition~\ref{prop:degenerate}, $A$ is not a degenerate $k/n$-base.
\end{proof}

%%%%%%%%
\subsection{\texorpdfstring{\boldmath $1/\omega$}{1/omega}-bases}
%%%%%%%%

A $1/n$-base is computable from each of the $n$ coordinates of some Martin-L\"of random $(Z_1,\dots, Z_n)$. One can generalise this to infinite sequences. We now work in the computable probability space $(2^\w)^\w$. It is effectively isomorphic to~$2^\w$ via a measure-preserving map. Such a map is determined by a computable partition of~$\w$ into infinitely many computable sets (``columns''). Below we will use the fact that this can be done in such a way that the density of each column is positive. 

\begin{defn}
A set $A$ is a \emph{$1/\omega$-base} if there is a Martin-L\"of random sequence $(Z_1,Z_2,\dots)$ such that $(\forall i)\; A\leq_\Tur Z_i$.
\end{defn}

Such bases are now easy to characterise.

\begin{prop}\label{prop:1/omega}
 A set is a $1/\omega$-base iff it is a $p$-base for every rational $p>0$.
\end{prop}
\begin{proof}
Assume that $A$ is a $1/\omega$-base witnessed by $(Z_1,Z_2,\dots)$. For each $n$, the sequence $(Z_1,\dots, Z_n)$ witnesses that $A$ is a $1/n$-base. This implies that $A$ is a $p$-base for every rational $p>0$

Now assume that $A$ is a $p$-base for every rational $p>0$. Consider breaking $\Omega$ up into countably many sequences $\{\bar \Omega_n\}_{n\in\omega}$ such that $\Omega = \Omega_1\oplus(\Omega_2\oplus(\Omega_3\oplus(\cdots)))$, where here~$\oplus$ is the usual split into evens and odds. In other words, $\bar \Omega_n$ is a $2^{-n}$-part of $\Omega$. For each $n$, we know that $A$ is a $2^{-n}$-base. Hence by \cref{thm:more_on_k_n_bases}, $A\leq_\Tur\bar \Omega_n$. Therefore, $A$ is a $1/\omega$-base as witnessed by the sequence $(\bar \Omega_1,\bar\Omega_2,\dots)$. 
\end{proof}

The proof shows that every $1/\omega$-base is witnessed by a single Martin-L\"of random sequence $(\bar \Omega_1,\bar \Omega_2,\dots)$ that arises from a computable partition of~$\Omega$. Again we remark that the proof used a partition of~$\omega$ into columns of positive density; if we use G\"odel's pairing function (as is commonly done), then each column has density 0 and the proof will not work. This distinction is only important when we consider the ways in which $\Omega$ can be considered as a universal witness for being a $1/\omega$-base; it does not affect the definition of being a $1/\omega$-base, in that a set~$A$ is a $1/\omega$-base if and only for some, or any, effective measure-preserving isomorphism $j\colon 2^\w\to (2^\w)^\w$ there is a random sequence $Z\in 2^\w$ such that $A$ is computable from each coordinate of $j(Z)$. 

 % Under the usual definition of infinite join, no column of $\Omega$ is a $p$-part of $\Omega$ for any $p>0$, so we cannot actually say that $\Omega$ witnesses every $1/\omega$-base. But this observation has very little content because we could have simply defined $1/\omega$-bases using the infinite join from in Proposition~\ref{prop:1/omega}; that would obviously not have changed the class of $1/\omega$-bases, just the meaning of ``witness''.

\smallskip

As mentioned in the introduction, the notion of a $1/\omega$-base could theoretically be weakened, but we obtain an equivalent notion. The proof of the first direction of \cref{prop:1/omega} shows:

\begin{proposition} \label{prop:weak_1_omega}
	A set~$A$ is a $1/\omega$-base if and only if there is a countable infinite set $Q\subset 2^\w$ such that: (a) every $Z\in Q$ computes~$A$; and (b) the join of any finitely many elements of~$Q$ is random.
\end{proposition}

\subsubsection{$1/\omega$-bases and strong jump-traceability} % (fold)

Recall that a set $A$ is $\omega$-c.a.\ if it can be computably approximated with a computably bounded number of changes; equivalently, $A \le_{wtt} \emptyset'$. A set is \emph{strongly jump-traceable} (SJT) if it is~$h$-jump traceable for every order function~$h$; see \cite{GreenbergTuretsky:SJTsurvey} for a survey. Every strongly jump traceable set is a $1/\omega$-base. For, by \cite{Greenberg.Hirschfeldt.ea:12} and \cite{DGT:InherentEnumerabilityofSJT} together, any SJT set is computable from every $\w$-c.a.\ random sequence; the columns of~$\Omega$ are $\w$-c.a.

On the other hand, there is an $1/\omega$-base that is not SJT. To see this, let $\cost = \sum_n \tp{-n}\cost_{\Omega, \tp{-n}}$. Then any set obeying $\cost$ is an $1/\omega$-base, and $\cost$ is a benign cost function in the sense of \cite{GreenbergNies:benign}. Thus there is a computable order~$h$ such that every $h$-jump traceable obeys~$\cost$ (ibid.). But by work of Ng \cite{Ng:08}, $h$-jump traceability is strictly weaker than SJT.

We conjecture that the $1/\omega$ bases form a $\Pi^0_4$ complete ideal.

%Every set obeying $\cost$ obeys all of its summands. And one can always find a benign cost function $\cost$ even more stringent and a set obeying $\cost$ but not $\cost'$, so not SJT by work of Selwyn Ng.
% \hl{ref? Or could this be in the benign paper?}
%
%Quick: Let~$c$ be benign. 

%An intersection of ideals is an ideal, so the $1/\omega$-bases also form an ideal. Is this ideal also generated by its c.e.\ elements?

%\hl{Leave this?}
%
%\begin{query}
%Is every $1/\omega$-base computable from a c.e.\ $1/\omega$-base?
%\end{query}

%%%%%%%%%
%%%%%%%%%
%\section{Smart sets}
%%%%%%%%%
%%%%%%%%%
%
%
%
%
%
%
%Question: if I'm a $k/n$-base, can the random witness be required to be incomplete?

%%%%%%%%
%%%%%%%%
\section{Robust computability from random sequences} 
\label{sec:robust}
%%%%%%%%
%%%%%%%%
%Recall that $A$ is robustly computed by $Z$ if $A \leT Y $ for each $Y$ such that $\rho(Y \lra Z)=1$.
%\begin{thm}[Andre Version] The following are equivalent for a set $A$.
%\begin{itemize}
%\item[(i)] $A$ is robustly computable from $\Omega$
%\item[(ii)] $A$ is robustly computable from some ML-random 
%\item[(iii)] $A$ is an $(n-1)/n$-base for some $n$. \end{itemize}
%\end{thm}
Recall that a set is \emph{robustly computable} from a sequence~$Z$ if it is computable from every~$Y$ such that the upper density of $Y\symdif Z$ is 0 (such a~$Y$ is called a ``coarse description'' of~$Z$). This notion was investigated in \cite{Hirschfeldt.Jockusch.ea:15}, where it is shown that every set that is robustly computable from a random sequence is $K$-trivial, and in fact, is a $(k-1)/k$-base for some~$k$.

In this section we provide the proof of the converse, Theorem~\ref{thm:Kuyper_class}, which states that 
the following are equivalent for a set~$A$:
	\begin{enumerate}
		\item $A\in \+{B}_{<1}$ (that is, $A$~is a $p$-base for some~$p<1$). 
		\item $A$ is robustly computable from some random sequence. 
		\item $A$ is robustly computable from~$\Omega$. 
		\item There is a $\delta>0$ such that~$A$ is computable from all sets~$Z$ such that the upper density of $Z\symdif \Omega$ is less than~$\delta$.
	\end{enumerate}

\begin{proof} (4)$\to$(3)$\to$(2) are trivial. As mentioned, the implication (2)$\to$(1) is in the proof of \cite[Thm.\ 3.2]{Hirschfeldt.Jockusch.ea:15}. 

It remains to show (1) $\to$ (4). Recall that for strings $\s,\tau$ of the same length~$n$, we let 
\[
	d(\s,\tau) = \frac{|\{ i \,:\, \sigma(i ) \neq \tau(i) \}|}{n}
\]
and that for $X,Y\in 2^\w$ we let $d(X,Y)=\limsup_n d(X\rest{n}, Y\rest{n})$. For all strings~$\s$ and all $q\in [0,1]$, we let $B(\s,q) = \left\{ \tau \,:\,  |\tau| = |\s|\andd d(\s,\tau)\le q \right\}$. A well-known estimate gives $|B(\s,q)| \le 2^{H(q)n}$ when $q\le 1/2$, where $H$ is binary entropy: $H(q) = - q \log_2 (q) - (1-q) \log_2(1-q)$ (see for example \cite[Cor.9,p.310]{MacWilliamsSloane}). Now choose $\delta$ small enough so that $H(2\delta) < 1-p$ (so if~$p$ is close to~1, then~$\delta$ will be small; if~$p$ is very small, then~$\delta$ can approach $1/4$). 

Let~$A$ be a $p$-base; take $Z\in 2^\w$ such that $d(\Omega,Z)<\delta$. We show that $Z$ computes~$A$. By our proof of \cref{thm:more_on_k_n_bases}, we may assume that~$A$ is c.e. 
For every string~$\tau$, let $G_\tau = \bigcup \left\{ [\rho] \,:\, \rho\in B(\tau,2\delta)  \right\}$. Then
\[
	\mu (G_\tau) = 2^{-|\tau|}|B(\tau,2\delta)| \le 2^{-|\tau|}2^{(1-p)|\tau|} = 2^{-p|\tau|}. 
\]

Fix $\s\prec Z$ such that for all $m\ge |\s|$, $d(Z\rest{m},\Omega\rest{m})\le \delta$. We define a functional~$\Gamma$ using an approximation $\seq{A_s}$ of~$A$ that witnesses that $A$ obeys $\cost_{\Omega,p}$. We also use an approximation of $\Omega$ such that $\Omega_{s+1}-\Omega_s \ge 2^{-s-1}$. We define~$\Gamma$ as follows: for every $n\ge |\s|$, for every string~$\tau$ of length~$n$ extending~$\s$, we set $\Gamma(\tau,n) = A_s(n)$ where~$s$ is the least stage $s> n$ such that for all $m\in [|\s|,n]$, $d(\tau\rest{m},\Omega_s\rest{m})\le \delta$ (note~$\delta$ and not $2\delta$). If there is no such stage~$s$, then $\Gamma(\tau,n)\diverge$. 

The assumption on~$Z$ implies that $\Gamma(Z,n)\converge$ for all $n\ge |\s|$. Let $s(n)$ be the stage at which the computation $\Gamma(Z\rest{n},n)$ is made. We need to show that for all but finitely many~$n\in A$, $n$ enters~$A$ by stage~$s(n)$. 

\smallskip

We enumerate open sets~$U_n$ for $n\ge |\s|$. If $n\notin A$ then $U_n = \emptyset$. If~$n\in A$, let~$t=t(n)$ be the stage at which~$n$ is enumerated into~$A$ (i.e., $n\in A_t\setminus A_{t-1}$). If $t(n)\le n$ then $U_n = \emptyset$. Suppose that $n<t(n)$. Resembling the proof of \cref{prop:covering_p-OW-tests_by_p-Auckland_tests}, let~$k = k_t$ be the unique~$k$ such that $2^{-k-1}\le \Omega_t-\Omega_n < 2^{-k}$. Note that by our choice of approximation to~$\Omega$, we have $k\le n$.  We let $U_n = \bigcup_{s=n}^{t(n)} G_{\Omega_s\rest{k}}$. Again, there are at most two values~$\rho$ for $\Omega_{s}\rest{k}$ for $s\in [n,t]$. Thus our calculation above shows that $\leb(U_n) \le 2 \cdot 2^{-pk}\le 2\cdot 2^p\cdot (\Omega_t-\Omega_n)^p$. Recall that the $\cost_{\Omega,p}$-cost of enumerating~$n$ into~$A$ is exactly $(\Omega_t-\Omega_n)^p$. Since~$A$ obeys~$\cost_{\Omega,p}$, we see that $\sum_n \leb(U_n)$ is finite, that is, $\seq{U_n}$ is a Solovay test. Thus, $\Omega\in U_n$ for only finitely many~$n$. 

Since $\Omega_s\to \Omega$, for all but finitely many~$n$, $\Omega-\Omega_n < 2^{-|\s|}$, which shows that for all but finitely many~$n\in A$, $k_{t(n)}\ge |\s|$. Let $n\ge|\s|$ and suppose that $\Gamma(Z,n) \ne A(n)$; so $n\in A$ and $s(n)< t(n)$. Suppose that $k_{t(n)}\ge |\s|$. We show that $\Omega\in U_n$. Note that $n<s(n)$ so $n<t(n)$. Let $k = k_{t(n)}$. Then $G_{\Omega_{s(n)}\rest{k}}\subseteq U_n$, and since $k\in [|\s|,n]$, we have $d(Z\rest{k},\Omega_{s(n)}\rest{k})\le \delta$. But by assumption on~$Z$, $d(Z\rest{k},\Omega\rest{k})\le \delta$. So $d(\Omega_{s(n)}\rest{k}, \Omega\rest{k})\le 2\delta$, i.e., $\Omega\rest{k}\in B(\Omega_{s(n)}\rest{k},2\delta)$, so $\Omega\in G_{\Omega_{s(n)}\rest{k}}$. 
\end{proof}

We remark that (2)$\to$(4) above is implied by \cite[Thm.\ 3.7]{Hirschfeldt.Jockusch.ea:15}, which is more general; the proof is more elaborate. 

% We also remark that \cite[Cor.\ 3.10]{Hirschfeldt.Jockusch.ea:15} states that for every $\Delta^0_2$ random sequence~$Z$ there is some noncomputable~$A$ which is robustly reducible to~$Z$; this result does not seem to follow from the work here. 

% *** How much uniformity is possible?

\medskip

Theorem 3.19 of \cite{Hirschfeldt.Jockusch.ea:15} states that not every~$K$-trivial set is robustly computable from a random sequence. This fact can now be established using cost functions. It is not difficult to construct a cost function~$\cost$ such that $\limcost >^\times \limcost_\Omega$ but for all~$p<1$, $\limcost <^\times \limcost_{\Omega,p}$. By \cref{prop:calculus_quote} there is a set~$A$ obeying~$\cost_\Omega$ but not~$\cost$. That set is $K$-trivial but not a $p$-base for any~$p<1$, hence not robustly reducible to a random. We extend this result a little. Say that an ideal~$\+{I}\subseteq \Delta^0_2$ is characterised by a cost function~$\cost$ if $\+{I}$ is the collection of sets that obey~$\cost$. 

\begin{proposition} \label{prop:robust_ideal_no_cost}
	The ideal $\+{B}_{<1}$ is not characterised by a cost function. 
\end{proposition}

In particular, $\+{B}_{<1}$ is not the ideal of $K$-trivial sets, as the latter is characterised by $\cost_\Omega$. \Cref{prop:robust_ideal_no_cost} gives the first example of a $\Sigma^0_3$ subideal of the $K$-trivial sets that is not characterised by any cost function.\footnote{Every $K$-trivial set is $\w$-c.a.\ and there is a uniform listing of all such sets (e.g., \cite[1.4.5]{Nies:book}). By a $\Sigma^0_3$-ideal we mean that the collection of sets in the ideal is a $\Sigma^0_3$ subset of Cantor space, or equivalently, that the collection of $\w$-c.a.\ indices of the elements of the ideal is $\Sigma^0_3$.}
%
%For the following lemma, for functions~$f$ and~$g$ we write $f \leneq^\times g$ to indicate that $f\le^\times g$ but $g\nle^\times f$ (note that this implies $f<^\times g$). \hl{What does $<^\times$ mean? Do we really need this complication?}
The next lemma, which is key to the proof of \Cref{prop:robust_ideal_no_cost}, essentially says that there is no greatest lower bound for a strictly descending uniform sequence of cost functions.
\begin{lemma} \label{lem:squeeze_lemma_for_cost_functions}
	Let $\seq{\dcost_n}$ be a computable sequence of cost functions such that $(\forall n)\; \dcost_{n+1}\le \dcost_{n}$. Let $\ecost$ be a cost function such that $(\forall n)\; \elimcost \le^\times \dlimcost_n$ and $\dlimcost_n \not \le^\times \elimcost$.
	 Then there is a cost function $\cost\ge \ecost$ such that $(\forall n)\;\limcost\le^\times \dlimcost_n$ and $\limcost \not \le^\times \elimcost$. 
\end{lemma}

\begin{proof}
	We define $\cost(x,s)$ by induction on~$s$, starting with $\cost(x,s)=0$ for all $x\ge s$. At stage~$s$ we let, for each~$n<s$, $k_s(n)$ be the least~$k$ such that $n\cdot \ecost(k,s) \le \dcost_n(k,s)$ (for all $s$ and~$n$, $k_s(n)\le s$ by one of our assumptions on cost functions). We then define, for $x<s$, 
	\[
		\cost(x,s) = \begin{cases*}
			\max \{\cost(x,s-1), n\cdot \ecost(x,s)\} & if $n$ is greatest such that $x = k_s(n)$; \\
			\max \{ \cost(x,s-1), \ecost(x,s)\} & if $x\ne k_s(n)$ for all $n<s$. 
		\end{cases*}
	\]
	The point is that for each~$n$, the set $\{ k_s(n)\,:\, s<\omega \}$ is bounded:
	since $\dlimcost_n \nle^\times \elimcost$ there is some~$k$ such that for almost all~$s$, $n\cdot \ecost(k,s) < \dcost_n(k,s)$. If $k = k_s(n)$ for infinitely many~$s$ then $n\cdot \elimcost(k) \le \limcost(k)$, so $\limcost \nle^\times \elimcost$. But this also shows that for all~$n$, for almost all~$k$, $\limcost(k)\le \max\{\elimcost(k), \dlimcost_n(k)\}$, so $\limcost \le^\times \dlimcost_n$ follows from $\elimcost\le^\times \dlimcost_n$. 
\end{proof}
%\hl{Neat} thanks

%\begin{proof}[Proof of \cref{prop:robust_ideal_no_cost}]
%	Suppose, for a contradiction, that $\+{B}_{<1}$	is characterised by the cost function $\ecost$. \Cref{prop:calculus_quote} implies that for all~$p$, $\elimcost\le^\times \limcost_{\Omega,p}$. Apply \cref{lem:squeeze_lemma_for_cost_functions} for~$\ecost$ and $\dcost_n = \cost_{\Omega,(n-1)/n}$ to obtain a cost function $\cost$. Then $\elimcost \le^\times \limcost \le^\times \dlimcost_n$ implies that~$\cost$ too characterises $\+{B}_{<1}$. However $\limcost\nle^\times \elimcost$ implies that there is some set obeying $\elimcost$ but not $\limcost$, a contradiction. 
%\end{proof}

\begin{proof}[Proof of \cref{prop:robust_ideal_no_cost}]
	Suppose, for a contradiction, that $\+{B}_{<1}$	is characterised by the cost function $\ecost$. \Cref{prop:calculus_quote} implies that for all~$p$, $\elimcost\le^\times \limcost_{\Omega,p}$. Apply \cref{lem:squeeze_lemma_for_cost_functions} for~$\ecost$ and $\dcost_n = \cost_{\Omega,(n-1)/n}$ to obtain a cost function $\cost$. Then $\limcost \le^\times \dlimcost_n$ for each $n$ implies that~$\cost$ characterises a class containing $\+{B}_{<1}$. However $\limcost\nle^\times \elimcost$, so by \cref{prop:calculus_quote} again, there is a set obeying $\elimcost$ but not $\limcost$, which is a contradiction. 
\end{proof}
% \hl{Fixed some stuff e.g. wrong letters} looks good

% \begin{question} 
% Is there a cost function $\cc$ such that for c.e.\ $A$. 
% \[ A \models \cc \,\,\LR\,\, A \text{ is robustly computed by some ML-random? }\]
% \end{question}

% We suspect that the answer is negative. 

%%%%%%%%
%%%%%%%%
\section{Being computable from all weakly LR-hard randoms} \label{sec:aed}
%%%%%%%%
%%%%%%%%

This section provides further evidence that the ideal $\+{B}_{<1}$ is, in a sense, much smaller than the ideal of $K$-trivials. We show that it is properly contained in the ideal of degrees which lie below every so-called weakly LR-hard sequence; the latter ideal itself is properly contained in the $K$-trivial degrees.

As mentioned earlier, a set~$X$ is \emph{LR-hard} if every~$X$-random sequence is 2-random. 

\begin{proposition} \label{prop:LR-hard-box_implies_K_trivial}
	If a set~$A$ is computable from all LR-hard random sequences, then it is $K$-trivial. 
\end{proposition}

\begin{proof}
Day and Miller \cite{DM:15} showed that if~$A$ is not $K$-trivial, then there is a random~$X$ that is not a density 1 point and yet does not compute~$A$. Such a random must be LR-hard~\cite{BienvenuEtAl:DenjoyDemuthDensity}.	
\end{proof}

It is still unknown whether every $K$-trivial is computable from every LR-hard random sequence. We say that $X$ is  \emph{weakly LR-hard} if every~$X$-random sequence is Schnorr random relative to $\Halt$.

\begin{proposition}
There is a $K$-trivial set that is not computable from all weakly LR-hard randoms. 
\end{proposition}

\begin{proof}
 Barmpalias, Miller, and Nies~\cite{Barmpalias.Miller.ea:12} have shown that $X$ is weakly LR-hard if and only if $\emptyset'$ is c.e.\ traceable by $X$: there is a computable bound $h$ such that each function  $f \leT \emptyset'$ has an $h$-bounded trace c.e.\ in $X$.    A c.e.\ set is array computable if and only if it  is   c.e.\ traceable, and it is known that such a set can be properly low$_2$. Hence, by pseudojump inversion for ML-random sets, there is a weakly LR-hard ML-random $\DII$ set $X$ that is properly high$_2$. 
 
 Every random set Turing above every  $K$-trivial  is not Oberwolfach random in the sense of \cite{BGKNT:16}.  Hence it is LR-hard, and in particular high. So not every  $K$-trivial is   computable from all weakly LR-hard randoms.
\end{proof}

For background, there are several results characterising sub-ideals of the $K$-trivials as those degrees computable from all random elements of some null $\Sigma^0_3$ class. One example is the class of strongly jump-traceable sets;
they are precisely the sets computable from all superhigh random sequences \cite{Greenberg.Hirschfeldt.ea:12,GreenbergTuretsky:SJTsurvey}. 
%Earlier on, in \cite{GreenbergNies:benign}, a weaker result had been shown: every strongly jump traceable set is computable from all LR-hard random sequences.  
 \Cref{thm:more_on_k_n_bases} implies that the ideal of $k/n$-bases is such a class: the collection of sets computable from the $n$-columns of $\Omega$. This notion is closely related to that of a \emph{diamond class}: the class of \emph{c.e.}\ sets computable from all random sequences is some null~$\Sigma^0_3$ class. This restriction to c.e.\ sets is often immaterial, since the ideals under discussion are generated by their c.e.\ elements. However, at times we need to work harder to show one implication for general sets. For example, proving \cref{prop:LR-hard-box_implies_K_trivial} for c.e.\ sets~$A$ does not require the work of Day and Miller; we can use the existence of an incomplete LR-hard random, which follows from pseudo-jump inversion for randoms.

We also remark that the ideal of sets computable from every JT-hard random (studied implicitly in \cite{GreenbergNies:benign}  and in more detail in \cite[Section 8.5]{Nies:book}) contains the ideal of sets below every weakly LR-hard random; the former though is not yet known to be properly contained in the $K$-trivials.

\begin{remark}\label{rem:Kjos adapt} 
	Recall  that $X$ is   LK-hard if $(\fa y)\; K^X(y) \lep K^{\emptyset'}(y)$.   A computable measure machine   is a prefix free machine $M$ such that $\leb \Opcl  {\dom M}$ (the measure of its domain) is a computable real \cite[3.5.14]{Nies:book}.   We say that $X$ is \emph{weakly LK-hard} if $(\fa y)\; K^X(y) \lep K_{M}(y)$ for each computable measure machine $M$ relative to $\emptyset'$. 

Kjos-Hanssen et al.\ \cite{Kjos.Miller.ea:11} have proved  that  every LR-hard  set    is LK-hard. An adaptation of their argument,  available in \cite[Section 2]{LogicBlog:15},  shows that every weakly LR-hard  set    is weakly LK-hard. 
\end{remark}

\begin{thm}\ 
\begin{enumerate}
  	\item Every set in $\+{B}_{<1}$ is computable from all weakly LR-hard random sets.
  	\item Some set not in~$\+{B}_{<1}$ is computable from all weakly LR-hard random sets.
  \end{enumerate} 
\end{thm}
\begin{proof} (1) 	
Fix $p<1$. We show that every $p$-base is computable from all weakly LR-hard random sets. Let $Z$ be weakly LR-hard.

We will build a discrete measure~$\nu$ such that $\nu(m)$ is $\DII$ uniformly in $m$ and $\sum_m \nu(m)$ is a computable real. Let $\alpha$ be the universal uniform left-c.e.\ discrete measure, namely, $\alpha^Z(n)$ is the chance that the standard universal prefix-free machine with oracle~$Z$ prints out $n$.  Since $Z$ is weakly LK-hard, by Remark~\ref{rem:Kjos adapt}, $\alpha^Z \ge^\times   \nu$; this uses  a slight  adaptation of the Coding Theorem  (e.g., \cite[Thm.\ 2.2.25]{Nies:book} or \cite[Thm.\ 3.9.4]{Downey.Hirschfeldt:book}). 

We view $\alpha$ as a function of two variables, and  let

\[
\alpha_s^X (\w) = \sum_{n\in \w} \alpha_s^X(n) = \Omega_s^X;
\]
Let $\leb$ denote  Lebesgue measure and $c$ the  counting measure on~$\w$. By Fubini's Theorem
\[
I_s = \int \alpha_s(X,n) d(\leb\times c) = \int \alpha_s^X(\w) d\leb			
\]
  So $I_s \le 1$ for each $s$.

%so , $I = I_\w\le 1$.  NO I_\w  please, this is undefined

Let $\gamma_n = 2^{(p-1)n}$. The point is that $\sum \gamma_n$ is finite and computable, and $2^{-pn}\gamma_n = 2^{-n}$. 
We define $\nu(m)$ as a $\DII$ real uniformly in $m$. We view $m$ as a code for a pair of numbers. The algorithm for defining $\nu$ is as follows: 
\begin{quote} If $I_s\in (k\cdot 2^{-n}, (k+1)\cdot 2^{-n}]$, let $\nu_s(n,t) = \gamma_n$, and $\nu_s(n,t') = 0$ for all $t'\ne t$, where~$t$ is the least stage such that $I_t > k\cdot 2^{-n}$. \end{quote}
The total weight of $\nu$ is $\sum \gamma_n < \infty$. 

Now fix a constant~$d$ such that $d \cdot \alpha^Z \ge \nu $. We define a $p$-OW test that succeeds on all the oracles~$X$ such that $d\cdot \alpha^X\ge\nu$. 
If $I_s\in (k2^{-n},(k+1)2^{-n}]$ then the $k\tth$ version of $U_n$ at stage $s$ is the collection of oracles~$X$ such that $d\cdot \alpha_s^X(n,t)\ge \gamma_n$, where~$t$ is the least stage such that $I_t>k\cdot 2^{-n}$. 

The measure of each version of~$U_n$ is at most $d\cdot 2^{-n}/\gamma_n = d\cdot 2^{-pn}$. This is because by convention, for all~$X$, $\alpha^X_t(n,t)=0$; so if $I_s\in (k2^{-n},(k+1)2^{-n}]$ then $\int \alpha_s^X(n,t)\,d\leb \le I_s-I_t \le 2^{-n}$.

If $A$ is a $p$-base, then $A$ obeys $\cost_{\Omega, p}$, and hence $A \leT Z$ (\cref{prop:covering_p-OW-tests_by_p-Auckland_tests} and \cref{prop:weak_obedience_and_computing}).

\bigskip

\noindent (2) We modify the construction  for (1). We choose a non-decreasing computable function~$h\colon [0,1]\to \RR$ such that $h(x)\ge x$ and:
\begin{itemize}
	\item $\sum_n 2^{-n}/ h(2^{-n})<\infty$; 
	\item For all $p<1$, $h(x)\le^\times x^p$; and
	\item For all $M>0$, $h(Mx)\le^\times h(x)$. 
\end{itemize}
For example we can choose $h(x) = x(\log x)^2$. 

We carry out  the construction above with $\gamma_n = 2^{-n}/h(2^{-n})$. This tells us that every LR-hard set can be captured by an $h$-OW test, namely a test $(G_\s,\alpha)$ as in \cref{def:p-OW-tests} but such that $\leb(G_\s) \le^\times h(2^{-|\s|})$. We then follow the proof of \cref{prop:covering_p-OW-tests_by_p-Auckland_tests} to see that every such test can be covered by a $\cost_{\Omega,h}$-test, namely a test $\seq{V_n}$ such that $\leb(V_n)\le^\times h(\Omega-\Omega_n)$. So every set that obeys $\cost_{\Omega,h}$ is computable from all weakly LR-hard random sequences. Since $\cost_{\Omega,h}\le^\times \cost_{\Omega,p}$ for all $p<1$, \cref{prop:robust_ideal_no_cost} implies that there is a set that obeys $\cost_{\Omega,h}$ but is not in $\+{B}_{<1}$. 
\end{proof}

% {Since $\nu(m)$ is $\DII$ uniformly in $m$, this proof actually shows that every $p$-base is below every ML-random $Z$ such that $\MLR^Z \sub \SR[\ES']$. See \cite{Barmpalias.Miller.ea:12} for background on the latter class. It is equivalent to $\ES'$ is c.e.\ traceable by $Z$,  for all sets $Z$.*** Not sure this remark is true. Have to check. 

%Does every super high $Z$ satisfy $\MLR^Z \sub \SR[\ES']$? No, because there is a super high jump traceable. Use Barmpalias 2012 paper "Tracing and domination"\hl{ANDRE}}

%
% References
%

%
%\bibliographystyle{plain}
%\bibliography{SubclassesReferences}

\begin{thebibliography}{10}

\bibitem{BD:14}
George Barmpalias and Rod~G. Downey.
\newblock Exact pairs for the ideal of the {$K$}-trivial sequences in the
  {T}uring degrees.
\newblock {\em J. Symb. Log.}, 79(3):676--692, 2014.

\bibitem{Barmpalias.Miller.ea:12}
George Barmpalias, Joseph~S. Miller, and Andr{\'e} Nies.
\newblock Randomness notions and partial relativization.
\newblock {\em Israel J. Math.}, 191(2):791--816, 2012.

\bibitem{BT:97}
Dimitris Bertsimas and John Tsitsiklis.
\newblock {\em Introduction to Linear Optimization}.
\newblock Athena Scientific, 1st edition, 1997.

\bibitem{CoveringProblemAnnouncement}
Laurent Bienvenu, Adam~R. Day, Noam Greenberg, Anton{\'{\i}}n Ku{\v{c}}era,
  Joseph~S. Miller, Andr{\'e} Nies, and Daniel Turetsky.
\newblock Computing {$K$}-trivial sets by incomplete random sets.
\newblock {\em Bull. Symb. Log.}, 20(1):80--90, 2014.

\bibitem{BDHMS:12}
Laurent Bienvenu, Adam~R. Day, Mathieu Hoyrup, Ilya Mezhirov, and Alexander
  Shen.
\newblock A constructive version of {B}irkhoff's ergodic theorem for
  {M}artin-{L}\"of random points.
\newblock {\em Inform. and Comput.}, 210:21--30, 2012.

\bibitem{BGKNT:16}
Laurent Bienvenu, Noam Greenberg, Anton{\'{\i}}n Ku{\v{c}}era, Andr{\'e} Nies,
  and Dan Turetsky.
\newblock Coherent randomness tests and computing the {$K$}-trivial sets.
\newblock {\em J. Eur. Math. Soc. (JEMS)}, 18(4):773--812, 2016.

\bibitem{BienvenuEtAl:DenjoyDemuthDensity}
Laurent Bienvenu, Rupert H{\"o}lzl, Joseph~S. Miller, and Andr{\'e} Nies.
\newblock Denjoy, {D}emuth and density.
\newblock {\em J. Math. Log.}, 14(1):1450004, 35, 2014.

\bibitem{BT:95}
B{\'e}la Bollob{\'a}s and Andrew Thomason.
\newblock Projections of bodies and hereditary properties of hypergraphs.
\newblock {\em Bull. London Math. Soc.}, 27(5):417--424, 1995.

\bibitem{BMN:16}
Vasco Brattka, Joseph~S. Miller, and Andr{\'e} Nies.
\newblock Randomness and differentiability.
\newblock {\em Trans. Amer. Math. Soc.}, 368(1):581--605, 2016.

\bibitem{Chaitin}
Gregory~J. Chaitin.
\newblock Nonrecursive infinite strings with simple initial segments.
\newblock {\em IBM Journal of Research and Development}, 21:350--359, 1977.

\bibitem{CGFS:86}
F.~R.~K. Chung, R.~L. Graham, P.~Frankl, and J.~B. Shearer.
\newblock Some intersection theorems for ordered sets and graphs.
\newblock {\em J. Combin. Theory Ser. A}, 43(1):23--37, 1986.

\bibitem{DM:15}
Adam~R. Day and Joseph~S. Miller.
\newblock Density, forcing, and the covering problem.
\newblock {\em Math. Res. Lett.}, 22(3):719--727, 2015.

\bibitem{D:75}
Osvald Demuth.
\newblock The differentiability of constructive functions of weakly bounded
  variation on pseudo numbers.
\newblock {\em Comment. Math. Univ. Carolinae}, 16(3):583--599, 1975.

\bibitem{DGT:InherentEnumerabilityofSJT}
David Diamondstone, Noam Greenberg, and Daniel Turetsky.
\newblock Inherent enumerability of strong jump-traceability.
\newblock {\em Trans. Amer. Math. Soc.}, 367(3):1771--1796, 2015.

\bibitem{Downey.Hirschfeldt:book}
Rod~G. Downey and Denis Hirschfeldt.
\newblock {\em Algorithmic randomness and complexity}.
\newblock Springer-Verlag, Berlin, 2010.
\newblock 855 pages.

\bibitem{LogicBlog:15}
Andre~Nies (editor).
\newblock Logic {B}log 2015.
\newblock Available at \url{http://arxiv.org/abs/1602.04432}, 2015.

\bibitem{PaperOnBalancedTests}
Santiago Figueira, Denis~R. Hirschfeldt, Joseph~S. Miller, Keng~Meng Ng, and
  Andr\'e Nies.
\newblock Counting the changes of random {$\Delta^0_2$} sets.
\newblock {\em J. Logic Computation}, 25:1073--1089, 2015.
\newblock Journal version of paper at CiE 2010.

\bibitem{FGMN:12}
Johanna N.~Y. Franklin, Noam Greenberg, Joseph~S. Miller, and Keng~Meng Ng.
\newblock Martin-{L}\"of random points satisfy {B}irkhoff's ergodic theorem for
  effectively closed sets.
\newblock {\em Proc. Amer. Math. Soc.}, 140(10):3623--3628, 2012.

\bibitem{FranklinNg:Difference}
Johanna N.~Y. Franklin and Keng~Meng Ng.
\newblock Difference randomness.
\newblock {\em Proc. Amer. Math. Soc.}, 139(1):345--360, 2011.

\bibitem{Gacs:EverySequence}
P{\'e}ter G{\'a}cs.
\newblock Every sequence is reducible to a random one.
\newblock {\em Inform. and Control}, 70(2-3):186--192, 1986.

\bibitem{Greenberg.Hirschfeldt.ea:12}
Noam Greenberg, Denis~R. Hirschfeldt, and Andr\'e Nies.
\newblock Characterizing the strongly jump-traceable sets via randomness.
\newblock {\em Adv. Math.}, 231(3-4):2252--2293, 2012.

\bibitem{GreenbergNies:benign}
Noam Greenberg and Andr{\'e} Nies.
\newblock Benign cost functions and lowness properties.
\newblock {\em J. Symbolic Logic}, 76(1):289--312, 2011.

\bibitem{GreenbergTuretsky:SJTsurvey}
Noam Greenberg and Dan Turetsky.
\newblock Strong jump-traceability.
\newblock {\em Bulletin of Symbolic Logic}, 24(2):147--164, 2018.

\bibitem{Hirschfeldt.Jockusch.ea:15}
Denis~R Hirschfeldt, Carl~G Jockusch, Rutger Kuyper, and Paul~E Schupp.
\newblock Coarse reducibility and algorithmic randomness.
\newblock {\em The Journal of Symbolic Logic}, 81(3):1028--1046, 2016.

\bibitem{HirschfeldtNiesStephan:UsingRandomSetsAsOracles}
Denis~R. Hirschfeldt, Andr{\'e} Nies, and Frank Stephan.
\newblock Using random sets as oracles.
\newblock {\em J. Lond. Math. Soc. (2)}, 75(3):610--622, 2007.

\bibitem{Kjos.Miller.ea:11}
Bj{\o}rn Kjos-Hanssen, Joseph~S. Miller, and Reed Solomon.
\newblock Lowness notions, measure, and domination.
\newblock {\em J. London Math. Soc.}, 84, 2011.

\bibitem{Kucera:85}
Anton{\'{\i}}n Ku{\v{c}}era.
\newblock Measure, {$\Pi^0_1$}-classes and complete extensions of {${\rm PA}$}.
\newblock In {\em Recursion theory week ({O}berwolfach, 1984)}, volume 1141 of
  {\em Lecture Notes in Math.}, pages 245--259. Springer, Berlin, 1985.

\bibitem{Kucera:PriorityFree}
Anton{\'{\i}}n Ku{\v{c}}era.
\newblock An alternative, priority-free, solution to {P}ost's problem.
\newblock In {\em Mathematical foundations of computer science, 1986
  ({B}ratislava, 1986)}, volume 233 of {\em Lecture Notes in Comput. Sci.},
  pages 493--500. Springer, Berlin, 1986.

\bibitem{LW:49}
Lynn~H. Loomis and Hassler Whitney.
\newblock An inequality related to the isoperimetric inequality.
\newblock {\em Bull. Amer. Math. Soc}, 55:961--962, 1949.

\bibitem{MacWilliamsSloane}
F.~J. MacWilliams and N.~J.~A. Sloane.
\newblock {\em The theory of error-correcting codes. {I}}.
\newblock North-Holland Publishing Co., Amsterdam-New York-Oxford, 1977.
\newblock North-Holland Mathematical Library, Vol. 16.

\bibitem{MMT:12}
Mokshay Madiman, Adam~W. Marcus, and Prasad Tetali.
\newblock Entropy and set cardinality inequalities for partition-determined
  functions.
\newblock {\em Random Structures Algorithms}, 40(4):399--424, 2012.

\bibitem{Ng:08}
Keng~Meng Ng.
\newblock On strongly jump traceable reals.
\newblock {\em Ann. Pure Appl. Logic}, 154(1):51--69, 2008.

\bibitem{Nies:LownessPropertiesAndRandomness}
Andr{\'e} Nies.
\newblock Lowness properties and randomness.
\newblock {\em Adv. Math.}, 197(1):274--305, 2005.

\bibitem{Nies:book}
Andr\'e Nies.
\newblock {\em Computability and randomness}, volume~51 of {\em Oxford Logic
  Guides}.
\newblock Oxford University Press, Oxford, 2009.

\bibitem{Nies:CalculusOfCostFunctions}
Andr{\'e} Nies.
\newblock Calculus of cost functions.
\newblock In {\em The Incomputable, Cooper, S. Barry and Soskova, Mariya I.,
  eds.}, pages 183--216. Springer, 2017.

\bibitem{RSV:02}
A.~Romashchenko, A.~Shen, and N.~Vereshchagin.
\newblock Combinatorial interpretation of {K}olmogorov complexity.
\newblock {\em Theoret. Comput. Sci.}, 271(1-2):111--123, 2002.
\newblock Kolmogorov complexity.

\bibitem{Solovay:manuscript}
Robert~M. Solovay.
\newblock Draft of paper (or series of papers) related to {C}haitin's work.
\newblock IBM Thomas J. Watson Research Center, Yorktown Heights, NY, 215
  pages, 1975.

\end{thebibliography}

\end{document}